\documentclass{article}

\usepackage{amsmath,amsfonts,amssymb,amsthm}
\usepackage{stmaryrd,mathrsfs}
\usepackage{microtype}
\usepackage[hidelinks]{hyperref}
\usepackage{breakurl}
\usepackage{mathtools}
\usepackage{authblk}
\usepackage[ruled,linesnumbered,vlined]{algorithm2e}
\usepackage[shortlabels]{enumitem}
\setlist{font=\normalfont,leftmargin=*}

\usepackage{cleveref}
\crefname{equation}{Eq.}{Eqs.}
\crefname{thm}{Theorem}{Theorems} 
\crefname{lem}{Lemma}{Lemmata}
\crefname{cor}{Corollary}{Corollaries}
\crefname{prop}{Proposition}{Propositions}
\crefname{defn}{Definition}{Definitions}
\crefname{rem}{Remark}{Remarks}
\crefname{exmp}{Example}{Examples}
\crefname{notation}{Notation}{Notations}
\crefname{setup}{Setup}{Setup}
\crefname{ques}{Question}{Question}
\crefname{quest}{Question}{Question}
 
\newtheorem{thm}{Theorem}[section]
\newtheorem{lem}[thm]{Lemma}
\newtheorem{cor}[thm]{Corollary}
\newtheorem{prop}[thm]{Proposition}
\newtheorem*{pple}{The Principal Principle}

\newtheorem{thmintro}{Theorem}

\theoremstyle{definition}
\newtheorem{defn}[thm]{Definition}
\newtheorem{exmp}[thm]{Example}
\newtheorem{setup}[thm]{Setup}
\newtheorem{quest}[thm]{Question}

\newtheorem{ques}{Question}

\newtheorem*{exmp-intro}{Example}

\theoremstyle{remark}
\newtheorem{rem}[thm]{Remark}
\newtheorem*{notation}{Notation}

\numberwithin{equation}{section}

\DeclareMathOperator{\lct}{lct}
\DeclareMathOperator{\fpt}{fpt}
\DeclareMathOperator{\ft}{c}
\DeclareMathOperator{\crit}{crit}
\DeclareMathOperator{\lce}{lce}
\newcommand{\ftb}{\ft^{\idealb}}
\newcommand{\critb}{\crit\!^\idealb}
\newcommand{\ftab}{\ft^{\idealb}(\ideala)}
\newcommand{\critab}{\crit^{\idealb}(\ideala)}
\newcommand{\critabp}{\crit^{\idealb_p}(\ideala_p)}
\newcommand{\ftabp}{\ft^{\idealb_p}(\ideala_p)}

\newcommand{\nuab}{\nu_{\ideala}^{\idealb}}
\newcommand{\muab}{\mu_{\ideala}^{\idealb}}

\newcommand{\up}[1]{\left\lceil #1 \right\rceil}
\newcommand{\down}[1]{\left\lfloor #1 \right\rfloor}
\newcommand{\fr}[1]{\left\{ #1 \right\}}

\newcommand{\seq}[1]{\left( #1 \right)}

\newcommand{\norm}[1]{\left\|{#1}\right\|}
 
\DeclarePairedDelimiter{\ideal}{\langle}{\rangle}
\newcommand{\idealm}{\mathfrak{m}}
\newcommand{\idealp}{\mathfrak{p}}
\newcommand{\ideala}{\mathfrak{a}}
\newcommand{\idealb}{\mathfrak{b}}
\newcommand{\idealB}{\mathfrak{B}}
\newcommand{\idealc}{\mathfrak{c}}
\newcommand{\frob}[2]{#1^{[#2]}}
\newcommand{\idealfct}{\mathfrak{J}}

\newcommand{\vv}[1]{\mathbf{{#1}}}
\newcommand{\vvv}[1]{\boldsymbol{{#1}}}

\newcommand{\vx}{x}
\newcommand{\vz}{z}
\newcommand{\vw}{w}

\newcommand{\kk}{\Bbbk}
\newcommand{\LL}{\mathbb{L}}
\newcommand{\FF}{\mathbb{F}}
\newcommand{\RR}{\mathbb{R}}
\newcommand{\RRpos}{\RR_{>0}}
\newcommand{\RRnn}{\RR_{\ge0}}
\newcommand{\CC}{\mathbb{C}}
\newcommand{\ZZ}{\mathbb{Z}}
\newcommand{\QQ}{\mathbb{Q}}
\newcommand{\QQpos}{\QQ_{>0}}
\newcommand{\QQnn}{\QQ_{\ge0}}
\newcommand{\NN}{\mathbb{N}}
\newcommand{\NNpos}{\mathbb{N}_{>0}}
\newcommand{\infint}{(\QQnn)_{p^\infty}}
\DeclareMathOperator{\Crit}{Crit}

\newcommand{\cf}{\emph{cf}.}
\newcommand{\eg}{e.g., }
\newcommand{\ie}{i.e., }

\newcommand{\iref}{\ref}
\DeclareMathOperator{\supp}{supp}
\renewcommand{\geq}{\geqslant}
\renewcommand{\leq}{\leqslant}
\renewcommand{\ge}{\geqslant}
\renewcommand{\le}{\leqslant}

\begin{document}
 
\title{Frobenius Powers}
\author[*]{Daniel J.~Hern\'andez}
\author[**]{Pedro Teixeira}
\author[*]{Emily E.\ Witt}
\renewcommand\Affilfont{\footnotesize}
\affil[*]{\textsl{Department of Mathematics, University of Kansas, Lawrence, KS~66045, USA}
}
\affil[**]{\textsl{Department of Mathematics, Knox College, Galesburg, IL~61401, USA}
}
\date{}
\maketitle

\begin{abstract}
This article extends the notion of a \emph{Frobenius power} of an ideal in prime characteristic to allow arbitrary nonnegative real exponents.  
These generalized Frobenius powers are closely related to test ideals in prime characteristic, and multiplier ideals over fields of characteristic zero.
For instance, like these well-known families of ideals, Frobenius powers also give rise to jumping exponents that we call \emph{critical Frobenius exponents}. 
In fact, the Frobenius powers of a principal ideal coincide with its test ideals, but Frobenius powers appear to be a more refined measure of singularities than test ideals in general. 
Herein, we develop the theory of Frobenius powers in regular domains, and apply it to study singularities, especially those of generic hypersurfaces.
These applications illustrate one way in which multiplier ideals behave more like Frobenius powers than like test ideals. 
\end{abstract}

\section{Introduction}

This article concerns the singularities of algebraic varieties, especially the relationship between the singularities of hypersurfaces and those of more general varieties.  Though our main interest is in the prime characteristic setting, we are motivated by the following well-known result from birational algebraic geometry over the complex numbers:  Let $\ideala$ be an ideal of a polynomial ring over $\CC$.  If $f \in \ideala$ is a general $\CC$-linear combination of some fixed generators of $\ideala$, then \[ \mathcal{J}(f^t) = \mathcal{J}(\ideala^t)\] for each parameter $t$ in the open unit interval \cite[Proposition~9.2.28]{lazarsfeld.positivity-II}.
This equality of \emph{multiplier ideals} immediately implies that the 
\emph{log canonical threshold} of such an $f \in \ideala$ equals the minimum of $1$ and the log canonical threshold $\lct(\ideala)$ of $\ideala$.  

An important special case is when $\ideala$ is the term ideal of $f$.  
In this case, the condition that $f$ is general can be expressed concretely: it suffices to take $f$ nondegenerate with respect to the Newton polyhedron of $\ideala$ \cite[Corollary~13]{howald.multiplier_ideals_of_generic_polynomials}.  
Consequently, the multiplier ideals and log canonical threshold of a general polynomial can be computed combinatorially from its term ideal \cite{howald.multiplier_ideals_of_monomial_ideals}.

At present, it is understood that there is an intimate relationship between birational algebraic geometry over fields of characteristic zero and the study of singularities in prime characteristic from the point of view of the Frobenius endomorphism. 
Therefore, it is natural to  ask whether those results relating the singularities of a generic element $f \in \ideala$ to those of $\ideala$ extend to the positive characteristic setting, after replacing
multiplier ideals and the log canonical threshold with their analogs, namely test ideals and the  $F$-pure threshold.
Unfortunately, several obstructions are encountered, even when $f$ is a polynomial over $\overline{\FF}_p$ and $\ideala$ is its term ideal.  
For example, in this case the test ideals $\tau(\ideala^t)$ are monomial ideals, and depend only on the Newton polyhedron of $\ideala$ and the parameter $t$, but not on the characteristic  \cite[Theorem~6.10]{hara+yoshida.generalization_TC_multiplier_ideals}.   
In contrast,  the test ideals $\tau(f^t)$ need not be monomial and typically depend on the characteristic, often in mysterious ways.
Thus, it is not surprising that counterexamples to prime characteristic versions of the above statements
abound.  

\begin{exmp-intro}
Let $\kk$ be a field, consider the monomial ideal \[ \ideala = \ideal{x^2, y^3} \subseteq \kk[x,y], \] and let $f$ be an arbitrary $\kk^{\ast}$-linear combination of the generators $x^2$ and $y^3$. 

When $\kk = \CC$, the log canonical threshold of $\ideala$ equals $5/6$.  Furthermore, in this case, each choice of $f$ is nondegenerate with respect to its Newton polyhedron, and so the log canonical threshold of $f$ always equals that of $\ideala$.  In fact, a stronger relation holds:  The multiplier ideals of $\ideala$ and of every such $f$ agree at all parameters in the unit interval \cite[Example~9]{howald.multiplier_ideals_of_generic_polynomials}.

However, when $\kk = \overline{\FF}_p$, the $F$-pure threshold of $\ideala$ also equals $5/6$, while the $F$-pure threshold of every $f \in \ideala$ is strictly less than $5/6$ whenever $p \equiv 5 \bmod 6$ (see \cite[Example~4.3]{mustata+takagi+watanabe.F-thresholds} or  \cite[Example~3.8]{hernandez.diag_hypersurf}). In particular, for such primes, there exist parameters $t$ in the unit interval at which the test ideal of $\ideala$ differs from that of each $f$.
\end{exmp-intro}

Bearing this in mind, we are interested in the following question:  
  
\begin{ques}
\label{quesA: motivating Q}
In prime characteristic, how are the Frobenius singularities of an ideal related to those of a generic element of the ideal?  For example, to what extent is a given Frobenius invariant (especially the test ideals, $F$-pure threshold, and other $F$-jumping exponents) of a polynomial determined by an \emph{intrinsic}, but possibly different, Frobenius invariant of its term ideal?  
\end{ques}

With this motivation,  in this article we develop a new theory of Frobenius singularities of pairs in prime characteristic that we call \emph{\textup{(}generalized\textup{)} Frobenius powers}.  This construction assigns to an ideal $\ideala$ of an $F$-finite regular domain $R$ of characteristic $p>0$, and a nonnegative real number $t$, an ideal $\frob{\ideala}{t}$ of $R$ called the \emph{$t^\text{th}$ Frobenius power of $\ideala$}.  This theory is only interesting when $\ideala$ is nonzero and proper; we impose this assumption for the remainder of the introduction.

As the nomenclature and notation suggests, our Frobenius powers coincide with the standard Frobenius powers and the 
Frobenius roots of \cite{blickle+mustata+smith.discr_rat_FPTs} when the parameter $t$ is an integral power of~$p$. 
Indeed, our Frobenius powers are defined in terms of those standard operations, mimicking the construction of test ideals in \cite{blickle+mustata+smith.discr_rat_FPTs}.
Not surprisingly, the resulting theory bears many formal similarities to those of test ideals and of multiplier ideals.  
For example, as in these other theories, the ideals $\frob{\ideala}{t}$ vary discretely with $t$.
We call the parameters at which $\frob{\ideala}{t}$ ``jumps''
the \emph{critical \textup(Frobenius\textup) exponents} of $\ideala$.   
The smallest such parameter is called the \emph{least critical exponent} of $\ideala$, and is denoted $\lce(\ideala)$.  
The least critical exponent is an analog of the $F$-pure threshold, and as such, may be regarded as a prime characteristic analog of the log canonical threshold.  

The Frobenius powers  $\frob{\ideala}{t}$ are contained in the test ideals $\tau(\ideala^t)$;
	equality holds if $\ideala$ is principal, but otherwise they may differ drastically.  
For example, as noted above, the test ideals of a monomial ideal are combinatorial in nature, do not depend on the characteristic, and do not distinguish a monomial ideal from its integral closure (in fact, this last property holds for all ideals \cite[Lemma~2.27]{blickle+mustata+smith.discr_rat_FPTs}).   
In contrast, the Frobenius powers of a monomial ideal turn out to depend strongly on the characteristic and, as recorded in \Cref{exmp: Comparison with integral closure}, can differ from those of its integral closure.  
This  suggests that the Frobenius powers may often be more refined measures of singularities than test ideals.

Though Frobenius powers and test ideals can differ, one unifying observation is that the generalized Frobenius powers of an \emph{arbitrary} ideal behave in many interesting ways like the test ideals of a \emph{principal} ideal.  In fact, it is exactly this heuristic principle (which we call the \emph{Principal Principle}) that allows us to address certain instances of our motivating \Cref{quesA: motivating Q}.

Our main results are largely of two flavors: in some, we work in a fixed ambient ring, and in others, we consider Frobenius powers from the point of view of reduction to prime characteristic (that is, we let the characteristic tend to infinity).  
We summarize some of these results below.

\begin{thmintro}[\cf\ \Cref{thm: principalization}, and \Cref{cor: principalization for lce,cor: alternate principalization}]  If $\ideala$ is an ideal of an $F$-finite domain $R$, and $f$ is a {very} general linear combination of generators for $\ideala$, then the test ideals $\tau(f^t)$, which agree with the Frobenius powers $\frob{\ideal{f}}{t}$,  
are determined by the Frobenius powers $\frob{\ideala}{t}$ for every parameter $t$ in the unit interval.  In particular, the $F$-pure threshold of $f$ equals the least critical exponent of $\ideala$.
\end{thmintro}

At the level of $F$-pure thresholds, we also obtain a related result in which the ``very general" condition is weakened, though at the expense of restricting to the local case.  

\begin{thmintro}[\cf\ \Cref{thm: generic linear combinations 2} and \Cref{cor: crit(a) = max(c(g): g in a)}]
\label{generic: TI}
Let $\ideala$ be an ideal of a polynomial ring over $\kk = \overline{\FF}_p$ vanishing at a point $P$.  Fixing generators $\ideala = \ideal{g_1, \ldots, g_m}$, we may regard a given $\kk$-linear combination of these generators as a point in $\kk^m$. Under this identification, the set of all such generic elements of $\ideala$ whose $F$-pure threshold at $P$ agrees with the least critical exponent of $\ideala$ at $P$ is a nonempty open set of $\kk^m$.  
\end{thmintro}

The foregoing results show that the least critical exponent of an ideal can always be realized as the $F$-pure threshold of a principal ideal.  Consequently,  properties enjoyed by $F$-pure thresholds of principal ideals, but that may fail for arbitrary ideals, are inherited by least critical exponents.  For example,  although every rational number in the open unit interval is the $F$-pure threshold of some ideal, all least critical exponents avoid the same ``forbidden intervals'' that $F$-pure thresholds of principal ideals avoid (see \Cref{cor: forbidden intervals}, and the discussion preceding it, for more details).

Some of the most compelling results in the analogy between test ideals and multiplier ideals persist when test ideals are 
replaced with Frobenius powers. 
For example, in the sense that the multiplier ideal is a ``universal test ideal,'' it is also a ``universal Frobenius power.''  Likewise, the relationship between the least critical exponent and the log canonical threshold has strong similarities to
 the relationship that the $F$-pure threshold has with the latter.

\begin{thmintro}[\cf\ \Cref{prop: basic properties of critical exponents}, and \Cref{limits of critical exponents: T,limits of Frobenius powers: T}] 
\label{limits: TI}
Let $\ideala$ be an ideal of the localization of a polynomial ring over $\QQ$ at a point, and let $\ideala_p$ denote its reduction modulo a prime $p$. 
If $t$ is a parameter in the open unit interval, then
		\[ \frob{\ideala_p}{t} = \tau(\ideala_p^{t}) = \mathcal{J}(\ideala^{t})_p \] 
		for all $p \gg 0$.
Moreover, $ \lce(\ideala_p) \leq \fpt(\ideala_p) \leq \lct(\ideala)$, and if $\lct(\ideala)\leq 1$, then 
		\[ \lim_{p \to \infty} \lce(\ideala_p) = \lim_{p \to \infty} \fpt(\ideala_p) = \lct(\ideala).\]
\end{thmintro}

\subsection*{Some questions}

The results described herein inspire several natural questions.    
For instance, in \Cref{limits: TI}, we saw that the Frobenius powers, the test ideals, and the reductions of multiplier ideals coincide at each parameter in the open unit interval, at least when $p$ is large enough.  Motivated by this, and by a well-known question on the relationship between test and multiplier ideals, we ask the following.

\begin{ques}
\label{question 1: q}
What is the significance of the discrepancy between the Frobenius power $\ideala_p^{[t]}$ and the test ideal $\tau(\ideala_p^t)$, when they happen to differ in low characteristic?  Furthermore,  do there exist infinitely many $p$ such that \[ \frob{\ideala_p}{t} = \tau(\ideala_p^{t}) = \mathcal{J}(\ideala^{t})_p\] for all parameters $t$ in the open unit interval?  
\end{ques}

Our next question is motivated by \Cref{generic: TI}, which tells us that the local $F$-pure threshold of a sufficiently general polynomial at a point coincides with the least critical exponent of its term ideal at that point.
Furthermore, the proof is constructive, giving an explicit description of what it means to be ``sufficiently general.''

\begin{ques}
Is it possible to relate the \emph{sufficiently general} condition in \Cref{generic: TI} to other common notions of generality (\eg \emph{Newton nondegeneracy})? 
\end{ques}

As noted throughout this introduction, the problem of understanding the Frobenius powers of monomial ideals is closely related to the problem of understanding the test ideals of polynomials.  Thus, we ask the following question.

\begin{ques}
\label{monomial-question: q}
Given a monomial ideal $\ideala$ over a field of characteristic zero, can one explicitly describe the Frobenius powers of its reductions modulo $p \gg 0$?  For instance, do they vary uniformly with the class of $p$ modulo some fixed integer?  Do their critical exponents satisfy the uniformity conditions described in \cite[Problems~3.7, 3.8, and~3.10]{mustata+takagi+watanabe.F-thresholds}?
\end{ques}

In the article \cite{hernandez+etal.frobenius_examples}, the authors determine the Frobenius powers of diagonal monomial ideals, and of powers of the homogeneous maximal ideal, and give an affirmative answer to all parts of \Cref{monomial-question: q} in each case.  Furthermore, the authors expect to address \Cref{monomial-question: q} in general, in an upcoming paper.

In prime characteristic, it is not hard to show that the test ideal of a polynomial at a parameter $t$ is a monomial ideal if and only if it agrees with the Frobenius power of its term ideal at $t$. 
This motivates the following question.

\begin{ques} 
When is the test ideal of a polynomial a monomial ideal?
For instance, do there exist conditions on a monomial ideal $\ideala$ over a field of characteristic zero guaranteeing that for any sufficiently general polynomial $f \in \ideala$,  the test ideals $\tau(f_p^t)$ are also monomial for all $t$ in the open unit interval?  
 \end{ques}

\subsection*{Outline} In \Cref{s: prelim}, we review our notation, and  basics on multinomial coefficients in prime characteristic.   \Cref{s: frobenius} is dedicated to constructing, and establishing the basic properties of, generalized Frobenius powers.  In \Cref{s: crits}, we define and study critical Frobenius exponents,  and compare them to $F$-jumping exponents.  In \Cref{s: principal principle}, we relate Frobenius powers of arbitrary ideals to test ideals of principal ideals, and derive many of the results mentioned above.  In addition, we also establish the discreteness and rationality of critical exponents in this section.  We investigate Frobenius powers as the characteristic tends to infinity in \Cref{s: behavior as p approaches infinity}, and present an algorithm for computing Frobenius powers of ideals in a polynomial ring in \Cref{s: algorithm}.

\section{Preliminaries} \label{s: prelim}

\subsection{Notations and conventions}

Throughout the paper, $p$ denotes a positive prime integer, and $q$ is shorthand for a number of the form $p^e$ for some nonnegative integer $e$.  All rings considered are commutative.  If $\ideala$ is an ideal in a ring of characteristic $p$, then \[ 
\frob{\ideala}{q} = \ideal{f^q: f\in \ideala} \] is the $q^\textrm{th}$ Frobenius power of $\ideala$.
	
Vectors are denoted by boldface lower case letters, and their components by the same letter in regular font; \eg $\vv{v}=(v_1,\ldots,v_n)$. 
In line with this, $\vv{0} = (0, \ldots, 0)$ and $\vv{1} = (1,\ldots, 1)$.  The taxicab norm  $\norm{\vv{v}}$ of a vector $\vv{v}$ is the sum of the absolute values of its components.
    	 Given vectors $\vv{u}$ and $\vv{v}$ of the same dimension, $\vv{u}<\vv{v}$ denotes componentwise inequality, 
	so $\vv{u}<\vv{v}$ if and only if $u_i < v_i$ for each~$i$. 
The non-strict inequality $\vv{u}\le \vv{v}$ is defined analogously. 

We adopt standard monomial notation:  if $x_1, \ldots, x_m$ are elements of a ring, and $\vv{u} \in \NN^m$, then \[ \vx^{\vv{u}} = x_1^{u_1}\cdots x_m^{u_m}.\]

A rational number that can be written in the form $k/p^n$, for some $k\in \ZZ$ and $n\in \NN$, 
			is called \emph{$p$-rational}.
The sets consisting of all nonnegative and all positive $p$-rational numbers are denoted by $\infint$ and $(\QQpos)_{p^\infty}$, respectively.
The fractional part of a real number $t$ is denoted by $\fr{t}$; that is, $\fr{t}=t-\down{t}$.

\subsection{Multinomial coefficients}

The multinomial coefficient associated to a vector $\vv{u} \in \NN^m$ is
	\[
		\binom{\norm{\vv{u}} }{\vv{u}}=\binom{\norm{\vv{u}}}{u_1,\ldots,u_m} = 
			\dfrac{\norm{\vv{u}}!}{u_1!\cdots u_m!}.
	\]
If $k$ is an integer with $k \ne \norm{\vv{u}}$, then we set $\binom{k}{\vv{u}}=0$.
	
\begin{thm}[\cite{dickson.multinomial}]
	Consider a vector $\vv{u} \in \NN^m$, and 
	write the terminating base $p$ expansions of $\norm{\vv{u}}$ and $\vv{u}$ as follows\textup:		
		\begin{equation*}
			\norm{\vv{u}} = k_0+k_1p+k_2p^2+\cdots+k_rp^r\quad \text{and} \quad
			\vv{u}=\vv{u}_0+p\,\vv{u}_1+p^2\,\vv{u}_2+\cdots+p^r\,\vv{u}_r,
		\end{equation*}
	where each $k_i$ is an integer between $0$ and $p-1$, inclusive, and each $\vv{u}_i$ is a vector in $\NN^m$ less than $p \cdot \vv{1}$.  	\textup{(}Note that it is possible that $k_r = 0$ or $\vv{u}_r = \vv{0}$.\textup{)}
	Then 
		\[
			\binom{\norm{\vv{u}}}{\vv{u}}\equiv \binom{k_0}{\vv{u}_0}\binom{k_1}{\vv{u}_1}\cdots \binom{k_r}{\vv{u}_r} \mod{p}.
		\]
	In particular, $\binom{\norm{\vv{u}}}{\vv{u}}\not\equiv 0\bmod{p}$ if and only if $\norm{\vv{u}_i}=k_i$ for each $i$, which is to say that the components of $\vv{u}$ sum to $\norm{\vv{u}}$ without carrying \textup(base $p$\textup).
\qed
\end{thm}

\begin{cor}\label{cor: multinomial congruence}
	Let $k,l\in \NN$, with $k<p$, and $\vv{u},\vv{v}\in \NN^m$, with $\vv{u}<p\cdot \vv{1}$.
	Then 
	\[
	\pushQED{\qed} 
	\binom{k+pl}{\vv{u}+p\vv{v}}\equiv \binom{k}{\vv{u}}\binom{l}{\vv{v}} \mod{p}.\qedhere
	\popQED
	\]
\end{cor}

\section{Frobenius powers}\label{s: frobenius}

\subsection{Integral powers}

Throughout this section, $\ideala$ is an ideal of a ring $R$ of prime characteristic $p$.

\begin{notation}
	Given $k\in \NN$ and $q$, a power of $p$, $\ideala^{k[q]}$ denotes the ideal $\frob{\big(\ideala^k\big)}{q}$ or, equivalently, 
		$\big(\frob{\ideala}{q}\big)^k$.
\end{notation}

\begin{defn}[Integral Frobenius power]
	Let $k\in \NN$, and write the base $p$ expansion of $k$ as follows: $k=k_0+k_1p+\cdots+k_rp^r$.
	Then the $k^\textrm{th}$ \emph{Frobenius power} of $\ideala$ is the ideal
		\[\frob{\ideala}{k}\coloneqq \ideala^{k_0}\ideala^{k_1[p]}\cdots \ideala^{k_r[p^r]}.\]
	If $k$ is a power of $p$, this agrees with the standard definition of Frobenius power.
\end{defn}

\begin{exmp}
	Since $p^s-1=(p-1)+(p-1)p+\cdots+(p-1)p^{s-1}$, we have
		\[\frob{\ideala}{p^s-1}=\ideala^{p-1}\ideala^{(p-1)[p]}\cdots \ideala^{(p-1)[p^{s-1}]}.\]
\end{exmp}

\begin{rem}\label{rem: characterization of Frobenius powers}
	The function $k\mapsto \frob{\ideala}{k}$ is the unique function $\idealfct$ from $\NN$ to the set of ideals of $R$ 
		satisfying the following properties\textup:
		\begin{enumerate}[(1)]
			\item\label{item: 0th power} 
				$\idealfct(0)=\ideal{1}$\textup;
			\item\label{item: k+pl}	
				$\idealfct(k+pl)=\ideala^k\cdot \frob{\idealfct(l)}{p}$, for each $k,l\in \NN$ with $k<p$.
		\end{enumerate}
\end{rem}

\begin{prop}[Basic properties of integral Frobenius powers]\label{prop: basic properties or integer powers}
	Let $\ideala$ and $\idealb$ be ideals of $R$, and $k,l\in \NN$.
	Then the following properties hold.
	\begin{enumerate}[(1)]
		\item\label{item: integral power vs regular} 
			$\frob{\ideala}{k}\subseteq \ideala^k$, and equality holds if $\ideala$ is principal.
		\item\label{item: integral power vs product} 
			$\frob{(\ideala\idealb)}{k}=\frob{\ideala}{k}\frob{\idealb}{k}$.  
		\item\label{item: integral power vs containment} 
			If $\ideala\subseteq \idealb$, then $\frob{\ideala}{k}\subseteq\frob{\idealb}{k}$\textup;
				consequently, $\frob{\ideala}{k}+\frob{\idealb}{k}\subseteq \frob{(\ideala+\idealb)}{k}$ and
				$\frob{(\ideala\cap \idealb)}{k}\subseteq \frob{\ideala}{k}\cap\frob{\idealb}{k}$.  
		\item\label{item: products of powers} 
			$\frob{\ideala}{k+l}\subseteq \frob{\ideala}{k}\frob{\ideala}{l}$, 
				and equality holds if $k$ and $l$ add without carrying \textup(base $p$\textup).
		\item\label{item: monotonicity of integral powers}
		 	If $k>l$, then $\frob{\ideala}{k}\subseteq \frob{\ideala}{l}$.  
		\item\label{item: powers of powers}
		 	$\frob{\ideala}{kl}\subseteq \frob{\big(\frob{\ideala}{k}\big)}{l}=\frob{\big(\frob{\ideala}{l}\big)}{k}$, 
				and the containment becomes an equality if one of the numbers $k$ and $l$ is a power of $p$.
		\item\label{item: F powers of regular powers}  
			$\frob{\ideala}{kl}\subseteq\big(\frob{\ideala}{k}\big)^l= \frob{\big(\ideala^l\big)}{k}$.
	\end{enumerate}
\end{prop}

\begin{proof}
	\begin{itemize}[wide=0pt]
		\item[\iref{item: integral power vs regular}--\iref{item: integral power vs containment}] follow directly from the definition of Frobenius powers.
		\item[\iref{item: products of powers}]
			That equality holds if $k$ and $l$ add without carrying (base $p$) is immediate from the definition of Frobenius powers.
			In general, write $k=k_1+pk_2$ and $l=l_1+pl_2$, where $0\le k_1,l_1<p$.
			If $k_1+l_1<p$, then \Cref{rem: characterization of Frobenius powers} shows that
				\[\frob{\ideala}{k+l}=\frob{\ideala}{k_1+l_1+p(k_2+l_2)}=\ideala^{k_1+l_1}\frob{\big(\frob{\ideala}{k_2+l_2}\big)}{p}.\]
			Induction on the sum $k+l$ allows us to assume that 
				$\frob{\ideala}{k_2+l_2}\subseteq \frob{\ideala}{k_2}\frob{\ideala}{l_2}$, and the desired containment follows easily.
			If $k_1+l_1\ge p$, then 
				\[\frob{\ideala}{k+l}=\frob{\ideala}{k_1+l_1-p+p(k_2+l_2+1)}=\ideala^{k_1+l_1-p}\frob{\big(\frob{\ideala}{k_2+l_2+1}\big)}{p}.\]
			Again, we may assume that $\frob{\ideala}{k_2+l_2+1}\subseteq \frob{\ideala}{k_2}\frob{\ideala}{l_2}\ideala$, so that
				\[
					\frob{\big(\frob{\ideala}{k_2+l_2+1}\big)}{p}
						\subseteq \frob{\ideala}{pk_2}\frob{\ideala}{pl_2}\frob{\ideala}{p}
						\subseteq \frob{\ideala}{pk_2}\frob{\ideala}{pl_2}\ideala^p,
				\]
				and the desired containment follows. 
		\item[\iref{item: monotonicity of integral powers}] 
			is a consequence of~\iref{item: products of powers}: if $k>l$, then 
				$\frob{\ideala}{k}\subseteq \frob{\ideala}{k-l}\frob{\ideala}{l}\subseteq \frob{\ideala}{l}$.
		\item[\iref{item: powers of powers}]
			The facts that $\frob{\big(\frob{\ideala}{k}\big)}{l}=\frob{\big(\frob{\ideala}{l}\big)}{k}$, and that this equals 
				$\frob{\ideala}{kl}$ when one of $k$ and $l$ is a power of $p$, follow immediately from the definition
				of Frobenius powers. 
			If $l<p$, repeated applications of~\iref{item: products of powers}
				gives us $\frob{\ideala}{kl}\subseteq (\frob{\ideala}{k})^l=\frob{\big(\frob{\ideala}{k}\big)}{l}$.
			If $l\ge p$, write $l=l_1+pl_2$, with $0\le l_1<p$.
			Then $\frob{\ideala}{kl}=\frob{\ideala}{kl_1+pkl_2}\subseteq \frob{\ideala}{kl_1}\frob{\big(\frob{\ideala}{kl_2}\big)}{p}$,
				by~\iref{item: products of powers}.
			The case already proven shows that $\frob{\ideala}{kl_1}\subseteq \frob{\big(\frob{\ideala}{k}\big)}{l_1}$, and 
				induction on~$l$ allows us to assume that 
				 $\frob{\ideala}{kl_2}\subseteq \frob{\big(\frob{\ideala}{k}\big)}{l_2}$;
				 the desired containment follows.
		\item[\iref{item: F powers of regular powers}]
			The containment $\frob{\ideala}{kl}\subseteq\big(\frob{\ideala}{k}\big)^l$ follows from~\iref{item: powers of powers} 
				and~\iref{item: integral power vs regular}, 
				or repeated applications of~\iref{item: products of powers}; the identity 
					$\big(\frob{\ideala}{k}\big)^l=\frob{\big(\ideala^l\big)}{k}$ follows from~\iref{item: integral power vs product}.
			\qedhere
	\end{itemize}
\end{proof}

\begin{prop}[Frobenius powers in terms of generators]\label{prop: generators for Frobenius powers} 
	Let $\ideala$ be an ideal of $R$, and let $\mathcal{A} \subseteq \ideala$ be a set of generators for~$\ideala$.
	Then the Frobenius power $\frob{\ideala}{k}$ 
		is generated by all products 
		$f^{\vv{u}}\coloneqq f_1^{u_1}\cdots f_m^{u_m}$, where $m \geq 1$, $f_i\in \mathcal{A} $, and $\vv{u}\in \NN^m$ is such that 
		$\norm{\vv{u}} = k$ and $\binom{k}{\vv{u}}\not\equiv 0\bmod{p}$.
\end{prop}

\begin{proof}
	For each $k$, let $\idealfct(k)$ be the ideal generated by the products $f^{\vv{u}}$, as in the statement.
	We shall verify that $\idealfct$ satisfies properties \iref{item: 0th power} and \iref{item: k+pl} of \Cref{rem: characterization of Frobenius powers}, 
		and the result will follow.
	Property \iref{item: 0th power} is trivial, so we focus on \iref{item: k+pl}.	
	Let $k,l\in \NN$, with $k<p$. 
	The ideal $\idealfct(k+pl)$ is generated by products $f^{\vv{w}}$, where
		$\norm{\vv{w}} = k + pl$ and
		$\binom{k+pl}{\vv{w}}\not\equiv 0 \bmod{p}$.
	Writing $\vv{w}=\vv{u}+p\vv{v}$, where $\vv{0}\le \vv{u}<p \cdot \vv{1}$, we have
		\[\binom{k}{\vv{u}}\binom{l}{\vv{v}}\equiv\binom{k+pl}{\vv{u}+p\vv{v}}\not\equiv 0 \mod{p},\]
		where the first congruence comes from  \Cref{cor: multinomial congruence}.		
	The generators $f^{\vv{w}}=f^{\vv{u}+p\vv{v}}$ of $\idealfct(k+pl)$ are, therefore, products of elements in the following sets: 
		\begin{align*}
			S&=\big\{f^\vv{u}:  f_1,\ldots,f_m\in \mathcal{A}, \vv{u}\in \NN^m, \norm{\vv{u}}=k, \textstyle{\binom{k}{\vv{u}}}\not\equiv 0\bmod{p}\big\};\\
			T &=\big\{(f^\vv{v})^p: f_1,\ldots,f_m\in \mathcal{A}, \vv{v}\in \NN^m, \norm{\vv{v}}=l, \textstyle{\binom{l}{\vv{v}}}\not\equiv 0\bmod{p}\big\}.
		\end{align*}
	Because $k<p$, $\binom{k}{\vv{u}}\not\equiv 0\bmod{p}$ for all $\vv{u} \in \NN^m$ for which $\norm{\vv{u}}=k$; thus, $S$ is simply a set of generators
		for $\ideala^k$.
	Since $T$  is a set of generators for $\frob{\idealfct(l)}{p}$, we conclude that $\idealfct(k+pl)=\ideala^k\cdot\frob{\idealfct(l)}{p}$.
\end{proof}

\begin{exmp}
	If $0<k<p$ and $q$ is a power of $p$, then if $i+j=kq-1$, we always have that
		$\binom{kq-1}{i,\, j}\not\equiv 0\bmod{p}$.
	Thus, the previous proposition 
		shows that if $\ideala$ is generated by two elements, then $\frob{\ideala}{kq-1}=\ideala^{kq-1}$.	
\end{exmp}

\begin{exmp}
	Let $f\in \kk[x_1,\ldots,x_n]$, where $\kk$ is a field of characteristic $p$.
	The \emph{support} of $f$, denoted by $\supp(f)$, is the collection of monomials that appear in $f$ with a nonzero coefficient.
	Then $\ideal{\supp(f^k)}\subseteq \ideal{\supp(f)}^k$, though this containment is typically strict.
	However, in view of \Cref{prop: generators for Frobenius powers}, we have
		$\ideal{\supp(f^k)}= \frob{\ideal{\supp(f)}}{k}$.
	This fact is one of our motivations for extending the notion of Frobenius powers.
\end{exmp}

\subsection{$p$-rational powers}

Henceforth we shall assume that $R$ is an $F$-finite regular domain of prime characteristic $p$. 
This will allow us to use the theory of $[1/q]^\mathrm{th}$ powers from  \cite{blickle+mustata+smith.discr_rat_FPTs}.
If $\ideala$ is an ideal of $R$ and $q$ a power of $p$, the ideal $\frob{\ideala}{1/q}$ is the smallest ideal $\idealc$ 
	such that $\ideala\subseteq \frob{\idealc}{q}$.
The following lemma gathers the basic facts about such powers, for the reader's convenience. 

\begin{lem}[{\cite[Lemma~2.4]{blickle+mustata+smith.discr_rat_FPTs}}]\label{lem: basic facts of roots}
	Let $\ideala$ and $\idealb$ be ideals of $R$. 
	Let $q$ and $q'$ be powers of $p$, and $k\in \NN$.
	Then the following statements hold.
	\begin{enumerate}[(1)]
		\item\label{item: root vs containment} 
			If $\ideala\subseteq \idealb$, then $\frob{\ideala}{1/q}\subseteq \frob{\idealb}{1/q}$.
		\item\label{item: root vs product} 
			$\frob{(\ideala\idealb)}{1/q}\subseteq \frob{\ideala}{1/q}\frob{\idealb}{1/q}$.
		\item\label{item: root of frob power} 
			$\frob{\big(\frob{\ideala}{q'}\big)}{1/q}=\frob{\ideala}{q'/q}\subseteq \frob{\big(\frob{\ideala}{1/q}\big)}{q'}$.
		\item\label{item: root of root} 
			$\frob{\ideala}{1/(qq')}= \frob{\big(\frob{\ideala}{1/q}\big)}{1/q'}$.
		\item\label{item: root vs frob power} 
			$\frob{\big(\frob{\ideala}{k}\big)}{1/q}\subseteq \frob{\big(\frob{\ideala}{1/q}\big)}{k}$. 
	\end{enumerate}
\end{lem}

\begin{proof}
	We prove \iref{item: root vs frob power} and the containment ``$\supseteq$'' in  \iref{item: root of root}; 
		proofs of the other properties can
		be found in \cite[Lemma~2.4]{blickle+mustata+smith.discr_rat_FPTs}.
	Property~\iref{item: root vs frob power} follows from the definition of $\frob{\ideala}{k}$ and repeated applications 
		of \iref{item: root vs product} and \iref{item: root of frob power}.		
	As for the reverse containment in \iref{item: root of root}, note that by  
		\iref{item: root of frob power} we have $\frob{\ideala}{1/q} =\frob{\ideala}{q'/(qq')}\subseteq \frob{\big(\frob{\ideala}{1/(qq')}\big)}{q'}$, 
		thus $\frob{\big(\frob{\ideala}{1/q}\big)}{1/q'}\subseteq \frob{\ideala}{1/(qq')}$, by the minimality of  $\frob{\big(\frob{\ideala}{1/q}\big)}{1/q'}$.
\end{proof}
	
\begin{defn}[$p$-rational Frobenius powers]
	Let $\ideala$ be an ideal of $R$.
	For each $k/q\in \infint$, we define
		\[\frob{\ideala}{k/q}\coloneqq \frob{\big(\frob{\ideala}{k}\big)}{1/q}.\]
\end{defn}
\noindent
Note that this definition is independent of the representation of $k/q$, since
	\[\frob{\big(\frob{\ideala}{pk}\big)}{1/(pq)}=\frob{\Big(\frob{\big(\frob{\ideala}{k}\big)}{p}\Big)}{1/(pq)}=\frob{\big(\frob{\ideala}{k}\big)}{1/q},\]
	where the first equality follows from \Cref{prop: basic properties or integer powers}\iref{item: powers of powers}, and the second from \Cref{lem: basic facts of roots}\iref{item: root of frob power}.
In particular, if $k/q$ is an integer, this coincides with the earlier definition.

\begin{lem}\label{lem: basic properties of rational powers} 
	Let $\ideala$ be a nonzero ideal of $R$ and $c,c'\in \infint$.
	\begin{enumerate}[(1)]	
		\item	\label{item: monotonicity} \textbf{\emph{(Monotonicity)}} 
			If $c'> c$, then $\frob{\ideala}{c'}\subseteq \frob{\ideala}{c}$.
		\item	\label{item: right continuity} \textbf{\emph{(Right constancy)}} 
			$\frob{\ideala}{c}= \frob{\ideala}{c'}$, for each $c'\ge c$ sufficiently close to $c$.
	\end{enumerate}
\end{lem}

\begin{proof}
	To prove \iref{item: monotonicity}, we may assume that $c$ and $c'$ have the same denominator $q$, 
		and use \Cref{prop: basic properties or integer powers}\iref{item: monotonicity of integral powers}
		and  \Cref{lem: basic facts of roots}\iref{item: root vs containment}. 
	As for~\iref{item: right continuity}, note that if $\frob{\ideala}{c+1/q}=\frob{\ideala}{c}$, for some $q$, then 
		$\frob{\ideala}{c}= \frob{\ideala}{c'}$ for each $c'\in \infint$ with $c\le c' \le c+1/q$, by \iref{item: monotonicity}.
	Thus, it suffices to prove the following claim:
	\begin{equation} \label{eqn: epsilon claim}
		\frob{\ideala}{c+1/q}= \frob{\ideala}{c} \text{ for some } q.
	\end{equation}	
	As $q$ increases, the ideals $\frob{\ideala}{c+1/q}$ form an ascending chain of ideals, which must eventually stabilize at some ideal $\idealb$. 
	We shall prove that $\frob{\ideala}{c}=\idealb$.
	Fix $q_0$ sufficiently large so that $cq_0\in \NN$ and $\frob{\ideala}{c+1/q_0}=\idealb$. 
	For each $q>1$ we have $\frob{\big(\frob{\ideala}{cqq_0+1}\big)}{1/(qq_0)}=\frob{\ideala}{c+1/(qq_0)}=\idealb$, so 
		\begin{equation*}
			\ideala\cdot\frob{\ideala}{cqq_0}=\frob{\ideala}{cqq_0+1}\subseteq \frob{\idealb}{qq_0}.
		\end{equation*}
	This shows that
		\[\ideala\subseteq  \big(\frob{\idealb}{qq_0}:\frob{\ideala}{cqq_0}\big)=\frob{\big(\frob{\idealb}{q_0}:\frob{\ideala}{cq_0}\big)}{q}\]
		for each $q>1$, where the equality is a consequence of the flatness of the Frobenius over $R$.
	If $\frob{\ideala}{cq_0}\not\subseteq\frob{\idealb}{q_0}$, the intersection of all 
		$\frob{\big(\frob{\idealb}{q_0}:\frob{\ideala}{cq_0}\big)}{q}$ would be the zero ideal
		by the Krull Intersection Theorem, contradicting the assumption that $\ideala$ is nonzero.
	Hence $\frob{\ideala}{cq_0}\subseteq \frob{\idealb}{q_0}$, and therefore 
		$\frob{\ideala}{c}=\frob{\big(\frob{\ideala}{cq_0}\big)}{1/q_0}\subseteq \idealb$.
	The reverse containment follows from \iref{item: monotonicity}, and \eqref{eqn: epsilon claim} holds.
\end{proof}

We close this subsection with the following lemma, which will be useful in proving properties of real Frobenius powers.

\begin{lem}\label{lem: qth roots vs. products}
	Let $\ideala$ and $\idealb$ be ideals of $R$. 
	Then $\frob{\big(\ideala \cdot \frob{\idealb}{q}\big)}{1/q}=\frob{\ideala}{1/q}\cdot \idealb$.
\end{lem}

\begin{proof}
It suffices to show that if $\idealc$ is an ideal of $R$, then $\ideala \frob{\idealb}{q} \subseteq \frob{\idealc}{q}$ if and only if $\frob{\ideala}{1/q} \cdot \idealb \subseteq \idealc$.  However, the containment $\ideala \frob{\idealb}{q} \subseteq \frob{\idealc}{q}$ holds if and only if 
\[ \ideala \subseteq ( \frob{\idealc}{q} : \frob{\idealb}{q} ) = \frob{( \idealc : \idealb )}{q}, \] 
which occurs if and only if $\frob{\ideala}{1/q} \subseteq (\idealc : \idealb)$, as desired.
\end{proof}

\subsection{Real powers}

We now extend the Frobenius powers to arbitrary nonnegative real exponents by using $p$-rational approximations from above.
Throughout, given a sequence $\seq{a_k}$ of real numbers, we use the notation  $a_k \searrow a$ to mean that the sequence 
	converges to the real number $a$ monotonically from above.
	
\begin{lem}\label{lem: independence of sequence}
	Let $\ideala$ be a nonzero ideal of $R$, $t\in \RRnn$, and suppose $\seq{t_k}$ and $\seq{s_k}$ 
		are sequences of $p$-rational numbers such that  $t_k\searrow t$ and $s_k\searrow t$.
	Monotonicity ensures that the Frobenius powers $\frob{\ideala}{t_k}$ and $\frob{\ideala}{s_k}$ form ascending chains of ideals, 
		which must eventually stabilize. 	
	Suppose $\frob{\ideala}{t_k}=\idealb$ and $\frob{\ideala}{s_k}=\idealc$, for all $k \ge k_0$.
	Then $\idealb=\idealc$.
\end{lem}

\begin{proof}
	If neither sequence is eventually constant, 
		choose $k_2,  k_1\ge k_0$ so that $t_{k_2}< s_{k_1}<t_{k_0}$; then   
		$\idealb=\frob{\ideala}{t_{k_0}}\subseteq \frob{\ideala}{s_{k_1}} =\idealc \subseteq \frob{\ideala}{t_{k_2}}=\idealb$,
		by \Cref{lem: basic properties of rational powers}\iref{item: monotonicity}.
	If $\seq{t_k}$ is eventually constant (so  $t$ is $p$-rational), by \Cref{lem: basic properties of rational powers}\iref{item: right continuity}
		we can choose $k\gg k_0$  sufficiently large so that $\frob{\ideala}{s_k}=\frob{\ideala}{t}$; then 
		$\idealc=\frob{\ideala}{s_k}=\frob{\ideala}{t}=\idealb$.
\end{proof}

\Cref{lem: independence of sequence} allows us to define 
	$\frob{\ideala}{t}$ for every $t\in \RRnn$ by taking $p$-rational approximations of $t$ from above. 

\begin{defn}[Real Frobenius powers]\label{defn: real frob power}
	Let $\ideala$ be a nonzero ideal of $R$.
	If $t\in \RRnn$ and $\seq{t_k}$ is a sequence of $p$-rational numbers such that $t_k\searrow t$, we define 
		\[\frob{\ideala}{t}\coloneqq \bigcup_{k\ge 0}\frob{\ideala}{t_k}.\]
	Thus, $\frob{\ideala}{t}=\frob{\ideala}{t_k}$,  for $k\gg 0$.
\end{defn}

Note that the definition of $\frob{\ideala}{t}$ does not depend on the choice of the sequence $\seq{t_k}$ by 
	\Cref{lem: independence of sequence}.	
For the purpose of computations of the ideals $\frob{\ideala}{t}$, we often use the sequence $\seq{\up{p^kt}/p^k}$.
If $t$ is $p$-rational, then by taking the constant sequence we see that this definition
	agrees with our earlier one.

\Cref{defn: real frob power} extends to the zero ideal, provided $t>0$.
Thus, $\frob{\ideal{0}}{t}=\ideal{0}$ for $t>0$, while $\frob{\ideal{0}}{0}=\ideal{0}^0=R$, as previously defined. 

If $\ideala$ is a principal ideal, then integral Frobenius powers of $\ideala$ are just regular powers; consequently, 
	the real Frobenius power $\frob{\ideala}{t}$ coincides with the test ideal $\tau(\ideala^t)$ of 
	\cite{blickle+mustata+smith.discr_rat_FPTs}, except in the case that $t=0$ and $\ideala$ is the zero ideal.

The following proposition generalizes \Cref{lem: basic properties of rational powers}, and will be used repeatedly throughout the paper 
	without further reference.
	
\begin{prop}
	Let $\ideala$ be a nonzero ideal of $R$, and $s,t\in \RRnn$.
	\begin{enumerate}[(1)]
		\item\label{item: monotonicity for real powers}	
			\textbf{\emph{(Monotonicity)}} If $s> t$, then $\frob{\ideala}{s}\subseteq \frob{\ideala}{t}$.
		\item\label{item: right constancy for real powers} 
			\textbf{\emph{(Right constancy)}} 
			$\frob{\ideala}{s}=\frob{\ideala}{t}$, for each $s>t$ sufficiently close to $t$. 
	\end{enumerate}
\end{prop}

\begin{proof}
	First, note that \iref{item: right constancy for real powers} follows from~\iref{item: monotonicity for real powers}:
		by definition, $\frob{\ideala}{t}=\frob{\ideala}{c}$ for some $p$-rational $c>t$, 
		and~\iref{item: monotonicity for real powers} implies that $\frob{\ideala}{t}=\frob{\ideala}{s}$ for each 
		$s$ between $t$ and $c$.
			
	Suppose $s>t$, and let $\seq{t_k}$ and $\seq{s_k}$ be sequences of $p$-rational numbers such that 
		$t_k\searrow t$ and $s_k\searrow s$.
	For $k\gg 0$ we have $s_k\ge s>t_k\ge t$, $\frob{\ideala}{s}=\frob{\ideala}{s_k}$, and 
		$\frob{\ideala}{t}=\frob{\ideala}{t_k}$, so that 
		\Cref{lem: basic properties of rational powers}\iref{item: monotonicity} yields~\iref{item: monotonicity for real powers}.
\end{proof}

\begin{rem}
	Right constancy fails at $t=0$ for the zero ideal, since $\frob{\ideal{0}}{0}=\ideal{0}^0=R$, while $\frob{\ideal{0}}{t}=\ideal{0}$ for each $t>0$.
\end{rem}

Additional properties of Frobenius powers are listed in the next proposition. 

\begin{prop}[Basic properties of real Frobenius powers]\label{prop: basic properties of real powers}
	Let $\ideala$ and $\idealb$ be ideals of $R$, and $t,s\in \RRnn$.
	Then the following properties hold.
	\begin{enumerate}[(1)]
		\item\label{item: real power vs test ideal} 
			If $t>0$ or $\ideala$ is nonzero, then $\frob{\ideala}{t}\subseteq \tau(\ideala^t)$, and equality holds if $\ideala$ is principal.
		\item 
			$\frob{(\ideala\idealb)}{t}\subseteq\frob{\ideala}{t}\frob{\idealb}{t}$.  		
		\item	
			If $\ideala\subseteq \idealb$, then $\frob{\ideala}{t}\subseteq\frob{\idealb}{t}$.
		\item\label{item: product of real powers}  
			$\frob{\ideala}{t+s}\subseteq \frob{\ideala}{t}\frob{\ideala}{s}$, and equality holds if $t<1$ and $s\in \NN$.
		\item\label{item: real power of real power}  
			$\frob{\ideala}{ts}\subseteq \frob{\big(\frob{\ideala}{t}\big)}{s}$, and equality holds if $t=p^e$ or $s=p^{-e}$, for some $e\in \NN$. 
	\end{enumerate}
\end{prop}

\begin{proof}
	These properties follow  from the properties of integral Frobenius powers  
		and of $[1/q]^\mathrm{th}$ powers (see  \Cref{prop: basic properties or integer powers} and 
		\Cref{lem: basic facts of roots}).
	To illustrate the methods, we verify \iref{item: product of real powers} and \iref{item: real power of real power}
		for a nonzero ideal $\ideala$.
	For each $k/q$ and $l/q$ in $(\QQnn)_{p^\infty}$, \Cref{prop: basic properties or integer powers}\iref{item: products of powers} and 
		 parts  \iref{item: root vs containment} and  \iref{item: root vs product} of
		\Cref{lem: basic facts of roots} give us 
		\begin{equation*}
			\frob{\ideala}{(k+l)/q}=
			\frob{\big(\frob{\ideala}{k+l}\big)}{1/q}\subseteq\frob{\big(\frob{\ideala}{k}\frob{\ideala}{l}\big)}{1/q}
				\subseteq\frob{\big(\frob{\ideala}{k}\big)}{1/q}\frob{\big(\frob{\ideala}{l}\big)}{1/q}
				=\frob{\ideala}{k/q}\frob{\ideala}{l/q}.
		\end{equation*}
	Letting $k/q\searrow t$ and $l/q\searrow s$ we obtain the  containment in~\iref{item: product of real powers}.
	If $s\in \NN$ and $t<1$, we may assume in this argument that $l=sq$ and $k/q<1$. 
	Since $k<q$, $k$ and $l=sq$ add without carrying (base $p$),
		and the first containment in the displayed equation becomes an equality, by 
		\Cref{prop: basic properties or integer powers}\iref{item: products of powers}.
	Furthermore, \Cref{lem: qth roots vs. products} shows that the second containment becomes an equality as well, since 
		$\frob{\ideala}{l}$ is a $[q]^\mathrm{th}$ power.
	So the additional assertion in~\iref{item: product of real powers} follows. 
	
	To verify \iref{item: real power of real power}, we again start with $k/q$ and $l/q$ in  $(\QQnn)_{p^\infty}$, and observe that 
		\[\frob{\ideala}{kl/q^2}=\frob{\big(\frob{\ideala}{kl}\big)}{1/q^2}\subseteq \frob{\Big(\frob{\big(\frob{\ideala}{k}\big)}{l}\Big)}{1/q^2}
		\subseteq \frob{\big(\frob{\ideala}{k/q}\big)}{l/q},\]
		where the first containment follows from \Cref{prop: basic properties or integer powers}\iref{item: powers of powers}
		and \Cref{lem: basic facts of roots}\iref{item: root vs containment}, and the second follows from 
		parts \iref{item: root vs containment}, \iref{item: root of root}, and \iref{item: root vs frob power} of \Cref{lem: basic facts of roots}.
	Letting $k/q\searrow t$ and $l/q\searrow s$ we obtain the  containment in~\iref{item: real power of real power}.

	If $t=p^e$, then for each $l/q$ we have 
		$\frob{\big(\frob{\ideala}{p^e}\big)}{l/q}=\frob{\big(\frob{\ideala}{p^el}\big)}{1/q}=\frob{\ideala}{p^el/q}$,
		by~\Cref{prop: basic properties or integer powers}\iref{item: powers of powers}.
	Letting $l/q\searrow s$ we see that $\frob{\big(\frob{\ideala}{t}\big)}{s}=\frob{\ideala}{ts}$.

	If $s=1/p^e$, then we have $\frob{\ideala}{t}\subseteq \frob{\big(\frob{\ideala}{t/p^e}\big)}{p^e}$, by the containment already proven.
	This is equivalent to $\frob{\big(\frob{\ideala}{t}\big)}{1/p^e}\subseteq\frob{\ideala}{t/p^e}$, or 
		$\frob{\big(\frob{\ideala}{t}\big)}{s}\subseteq \frob{\ideala}{ts}$.
\end{proof}

The following corollary is a rephrasing of the second assertion in 
	\Cref{prop: basic properties of real powers}\iref{item: product of real powers}---the 
	analog of Skoda's Theorem \cite[Proposition~2.25]{blickle+mustata+smith.discr_rat_FPTs} in our setting.
This result illustrates one way in which real Frobenius powers behave like test ideals of \emph{principal} ideals, 
	since Skoda requires that $\ideala$  
	can be generated by at most $m$ elements in order to conclude that $\tau(\ideala^{m+k}) = \ideala^{k+1} \tau(\ideala^{m-1})$ for $k \geq 0$.
	
\begin{cor}\label{cor: skoda}
	$\frob{\ideala}{t}=\frob{\ideala}{\down{t}}\cdot \frob{\ideala}{\fr{t}}$, for each $t\in \RRnn$, where $\fr{t}$ denotes the fractional
		part of $t$, \ie $\fr{t}=t-\down{t}$.
	\qed
\end{cor}

\begin{rem}
	We choose to work in an integral domain for simplicity, and because our main applications will be in the setting 
		of polynomial rings over $F$-finite fields.
	However, the notions and results introduced in this paper extend to arbitrary $F$-finite regular rings of positive characteristic, 
		provided one exerts care when dealing with $0^\textrm{th}$ powers. 
	For when $R$ is not a domain, right constancy at $t=0$ not only fails for the zero ideal, but also 
		for certain nonzero ideals.
	For instance, if $R=S\times S$, where $S$ is a regular $F$-finite domain of positive characteristic, 
		and $\ideala=\ideal{0}\times S$, 
		then $\frob{\ideala}{0}=\ideala^0=R$, while $\frob{\ideala}{s}=\ideala$, for each $s>0$.  
	The analogous issue is avoided in \cite{blickle+mustata+smith.discr_rat_FPTs} by virtue of their definition of 
		$\tau(\ideala^0)$; see \cite[Remark~2.10]{blickle+mustata+smith.discr_rat_FPTs}.
\end{rem}

\begin{rem}
	The real Frobenius powers commute with localization and completion; \ie the following hold
	for each $t\in \RRnn$:
	\begin{enumerate}[(1)]
		\item  If $S$ is a multiplicative system in $R$, then $\frob{(S^{-1}\ideala)}{t}  = S^{-1} (\frob{\ideala}{t})$.
		\item If $R$ is local and $\widehat{R}$ is its completion, then $\frob{(\ideala \widehat{R})}{t} = \frob{\ideala}{t} \widehat{R}$.
	\end{enumerate}
	These both follow from the fact that localization and completion commute with standard Frobenius and 
	$[1/q]^\textrm{th}$ powers of ideals (the last statement can be found in \cite[Lemma~2.7]{blickle+mustata+smith.discr_rat_FPTs}).
\end{rem}

We note that there exists an algorithm for computing $\frob{\ideala}{t}$ when $t$ is a nonnegative rational number and $\ideala$ is an ideal of a polynomial ring over a finite field.
See \Cref{s: algorithm} for details.

We conclude this section with a comparison of Frobenius powers and test ideals.  
We begin with \Cref{repeated_pigeonhole: L} below, which is well known to experts (see, \eg 		
	\cite[Lemma~3.2]{stefani+betancourt+perez.existence-of-FTs}).

\begin{lem}\label{repeated_pigeonhole: L}
	If $\ideala$ can be generated by $m$ elements, then 
		\[ \ideala^{(m + k)q} = \ideala^{(m-1)q}  \ideala^{(k+1)[q]}\]  
	for every integer $k \geq -1$.
\end{lem}

\begin{proof}  
	We induce on $k$, with the base case $k=-1$ being obvious, and the induction step following from the identity 
 		\[\ideala^{m q} = \ideala^{(m-1)q} \frob{\ideala}{q},\] 
 	which itself is a direct consequence of the pigeonhole principle. 
\end{proof}

\begin{lem}\label{uniform comparison: L}
	If $\ideala$ can be generated by $m$ elements, then 
		\[ \ideala^{ \frac{(m-1)(q-1)}{p-1} }  \ideala^k = \ideala^{ \frac{(m-1)(q-1)}{p-1} }   \frob{\ideala}{k} \] 
	for every $q$ a power of $p$, and every integer $k$ with $0 \leq k < q$. 
\end{lem}

\begin{proof} 
	\Cref{repeated_pigeonhole: L} implies that
		\[ \ideala^{(m-1)p^e} \ideala^{k_e p^e} =  \ideala^{(m-1)p^e} \ideala^{k_e [p^e]}, \]
		where $k_{e}$ is the coefficient of $p^e$ in the base $p$ expansion of $k$. 
	The bound $0 \leq k < q$ implies that $k_e = 0$ whenever $p^e$ is at least $q$, and so our claim follows after recalling that the 
		sum of all $p^e$ less than $q$ is $\frac{q-1}{p-1}$.  
\end{proof}

\begin{prop}\label{uniform relation between ideals: P}
 	If $\ideala$ can be generated by $m$ elements and $0 < t < 1$, then 
		\[ \tau(\ideala^{t + \frac{m-1}{p-1}}) \subseteq \frob{\ideala}{t} \subseteq \tau(\ideala^{t}).\]
\end{prop}

\begin{proof}
	Since $t < 1$, $\lceil t q \rceil < q$ for all $q \gg 0$, and for such $q$, 
		\[ \ideala^{ \left \lceil \left (t+ \frac{ m-1}{p-1}\right )q \right \rceil} \subseteq \ideala^{ \frac{(m-1)(q-1)}{p-1} + \lceil t q \rceil }  		\subseteq \ideala^{ \left [ \lceil t q \rceil \right ]}, \]
		where the last containment follows from \Cref{uniform comparison: L}. 
	Taking $[1/q]^\mathrm{th}$ powers of this and letting $q \to \infty$ then 
		shows that $\tau(\ideala^{t + \frac{m-1}{p-1}}) \subseteq \frob{\ideala}{t}$. 
\end{proof}

\subsection{Examples}
\label{subs:  Examples}

As noted in \Cref{prop: basic properties of real powers}\iref{item: real power vs test ideal}, the Frobenius power 
	$\frob{\ideal{f}}{t}$ agrees with the test ideal $\tau(f^{t})$ whenever $t > 0$.  
In this subsection, we compare Frobenius powers and test ideals of certain non-principal monomial ideals.  
In contrast with the principal setting, we will see that, even in this simple case, Frobenius powers and test ideals can be quite different.

We begin by recalling the situation for test ideals of monomial ideals:  
	If $\ideala$ is a monomial ideal in $\kk[x_1, \ldots, x_n]$, then $\tau(\ideala^{t})$ is generated by all 
	monomials $x^{\vv{u}}$ such that $\vv{u} + \vv{1}$ is contained in the interior of $t N$, where 
	$\vv{1} = (1, \ldots, 1) \in \ZZ^n$ and $N$ is the Newton polyhedron of $\ideala$ in $\RR^n$ 
	(see \cite[Theorem~6.10]{hara+yoshida.generalization_TC_multiplier_ideals}, and 
	\cite{howald.multiplier_ideals_of_monomial_ideals} for the analogous result for multiplier ideals).  
In particular, $\tau(\ideala^{t})$ depends on $N$ and $t$, but not on the ideal $\ideala$, nor on the characteristic.  
As we see below, this concrete description allows one to explicitly compute the test ideals of a monomial ideal without too much effort.

\begin{exmp}\label{exmp:  test ideals of m^7}
		Set $\idealm= \ideal{ x, y } \subseteq \kk[x,y]$.  If $\ideala = \idealm^7$  and $t \in [0,1)$, then
		\[\tau(\ideala^{t})=
		\begin{cases} 
			\kk[x,y] &\text{if } t \in \left[0, \frac{2}{7} \right) \\ 		
			\idealm &\text{if } t \in \left[\frac{2}{7}, \frac{3}{7}  \right) \\ 	
			\idealm^2 &\text{if } t \in \left[ \frac{3}{7} , \frac{4}{7} \right) \\ 		
			\idealm^3 &\text{if } t \in \left[ \frac{4}{7} , \frac{5}{7} \right) \\ 
			\idealm^4 &\text{if } t \in \left[ \frac{5}{7} , \frac{6}{7} \right) \\ 
			\idealm^5 &\text{if } t \in \left[ \frac{6}{7} , 1 \right)
		\end{cases}\]
We stress that these formulas are valid in all characteristics.
\end{exmp}

In the article \cite{hernandez+etal.frobenius_examples}, the authors compute the critical exponents and Frobenius powers of certain $\idealm$-primary monomial ideals in a polynomial ring.  
We summarize our computations in the context of \Cref{exmp:  test ideals of m^7} below. 

\begin{exmp}
	Adopt the notation of \Cref{exmp:  test ideals of m^7}.  As in that example, 
		\[ \{ \frob{\ideala}{t} : 0 \leq t < 1  \} = \{ \kk[x,y],  \idealm, \idealm^2, \idealm^3, \idealm^4, \idealm^5 \}. \] 
	However, the values of $t$ that correspond to each of these possibilities may depend on the characteristic.  
	For example, if $p \equiv 4 \bmod 7$, then
		\[\frob{\ideala}{t}=
		\begin{cases} 
			\kk[x,y] &\text{if } t \in \big[0, \frac{2}{7} - \frac{1}{7p} \big) \\ 		
			\idealm &\text{if } t \in \big[\frac{2}{7} - \frac{1}{7p}, \frac{3}{7}  \big) \\ 	
			\idealm^2 &\text{if } t \in \big[ \frac{3}{7} , \frac{4}{7} - \frac{1}{7p^2} \big) \\ 		
			\idealm^3 &\text{if } t \in \big[ \frac{4}{7} - \frac{1}{7p^2} , \frac{5}{7} \big) \\ 
			\idealm^4 &\text{if } t \in \big[ \frac{5}{7} , \frac{6}{7} \big) \\ 
			\idealm^5 &\text{if } t \in \big[ \frac{6}{7} , 1 \big)
		\end{cases}\]
	We note that there are similar formulas for every possible congruence class of $p$ modulo $7$.  
	For instance, if $p \equiv 1 \bmod 7$, then $\frob{\ideala}{t} = \tau(\ideala^{t})$ for every $t \in [0,1)$.  
	On the opposite extreme, if $p \equiv 3 \bmod 7$, then each of the intervals in $[0,1)$ that correspond to some 
		fixed value of $\frob{\ideala}{t}$ depends on the characteristic.
\end{exmp}

As mentioned above, the formulas for the test ideals of a monomial ideal depend only on the parameter $t$ and the Newton polyhedron of the ideal.    As it turns out, this is a particular instance of a more general fact:  the test ideals of an ideal are the same as those of its integral closure \cite[Lemma~2.27]{blickle+mustata+smith.discr_rat_FPTs}.  
Given this, it is natural to ask whether the analogous property also holds for Frobenius powers of ideals.  However, as we see below, the Frobenius powers of an ideal and its integral closure can differ, even in the monomial case.

\begin{exmp}\label{exmp: Comparison with integral closure} 
		Set $\idealm= \ideal{ x, y } \subseteq \kk[x,y]$.  If $\idealb = \idealm^5$  and $\idealc = \ideal{ x^5, y^5 }$, then it is clear that both these ideals determine the same Newton polyhedron in $\RR^2$, and so their test ideals must agree at all parameters.  In fact, 
				\[\tau(\idealb^{t})= \tau(\idealc^{t}) = 
		\begin{cases} 
			\kk[x,y] &\text{if } t \in \left[0, \frac{2}{5} \right) \\ 		
			\idealm &\text{if } t \in \left[\frac{2}{5}, \frac{3}{5}  \right) \\ 	
			\idealm^2 &\text{if } t \in \left[ \frac{3}{5} , \frac{4}{5} \right) \\ 		
			\idealm^3 &\text{if } t \in \left[ \frac{4}{5} , 1 \right)
					\end{cases}\]

On the other hand, when $p \equiv 3 \bmod 5$, we have computed that
	\[\frob{\idealb}{t}=
		\begin{cases} 
			\kk[x,y] &\text{if } t \in \big[0, \frac{2}{5} - \frac{1}{5p} \big) \\ 		
			\idealm &\text{if } t \in \big[\frac{2}{5}  - \frac{1}{5p} , \frac{3}{5}  - \frac{1}{5p^3}  \big) \\ 	
			\idealm^2 &\text{if } t \in \big[ \frac{3}{5}  - \frac{1}{5p^3} , \frac{4}{5}  - \frac{1}{5p^2} \big) \\ 		
			\idealm^3 &\text{if } t \in \big[ \frac{4}{5}  - \frac{1}{5p^2} , 1 \big)
			\end{cases}\]
while
	\[\frob{\idealc}{t}=
		\begin{cases} 
			\kk[x,y] &\text{if } t \in \big[0, \frac{2}{5} - \frac{1}{5p}  \big) \\ 		
			\idealm &\text{if } t \in \big[\frac{2}{5} - \frac{1}{5p}, \frac{3}{5} - \frac{2}{5p^2} \big) \\ 	
			\idealm^2 &\text{if } t \in \big[ \frac{3}{5} - \frac{2}{5p^2} , \frac{4}{5} - \frac{2}{5p} \big) \\ 		
			\ideal{ x^3, xy, y^3 } &\text{if } t \in \big[  \frac{4}{5} - \frac{2}{5p}  , \frac{4}{5} - \frac{1}{5p^2} \big) \\ 		
			\idealm^3 &\text{if } t \in \big[\frac{4}{5} - \frac{1}{5p^2} , 1 \big)
			\end{cases}\]
\end{exmp}

\begin{rem}  
	The somewhat unexpected monomial ideal $\ideal{x^3, xy, y^3}$ appearing in the formula for $\frob{\idealc}{t}$ above 
		is perhaps best realized as 
		\[ \ideal{ x^3, xy, y^3 } = \ideal{ x, y^3 } \cap \ideal{ x^3, y}. \] 
	Comparing this with the primary decomposition 
		\[ \idealm^3 = \ideal{ x, y^3 } \cap \ideal{ x^3, y } \cap \ideal{ x^2, y^2 }, \]
		the computation in \Cref{exmp: Comparison with integral closure} seems to suggest that the Frobenius powers of 
		$\idealc$ distinguish between certain irreducible components of the Frobenius powers of its integral closure $\idealb$.  
	We stress that, though this phenomenon is often reflected in our computations, the exact way in which the Frobenius 
		powers of an ideal and of its integral closure may differ remains quite mysterious.
\end{rem}

\section{Critical Frobenius exponents} \label{s: crits}

In this section, we define and investigate an analog of $F$-thresholds called \emph{critical Frobenius exponents}.   
Throughout this section, $\ideala$ and $\idealb$ are nonzero proper ideals of an $F$-finite regular domain 
	$R$ of characteristic $p>0$, with $\ideala\subseteq \sqrt\idealb$.

\subsection{Definition and basic properties}

For the reader's convenience, we recall here the definition and a few basic properties of $F$-thresholds.
We refer the reader to \cite{mustata+takagi+watanabe.F-thresholds} 
	and \cite{blickle+mustata+smith.discr_rat_FPTs} for a more detailed discussion.
For each~$q$, we define  $\nuab(q)=\max\{k\in \NN: \ideala^{k}\not\subseteq\frob{\idealb}{q}\}$.
As  $q$ varies over all powers of $p$, $\seq{\nuab(q)/q}$ forms a non-decreasing bounded sequence.
The \emph{$F$-threshold of $\ideala$ with respect to $\idealb$} is defined as
	\begin{equation*}
		\ftb(\ideala)=\lim_{q\to\infty}\frac{\nuab(q)}{q}=\sup_q \frac{\nuab(q)}{q}.
	\end{equation*}
The  $F$-threshold $\ftb(\ideala)$ is always a positive real number, which can be alternatively characterized 
	as follows:
	\begin{equation*}
		\ftb(\ideala)=\sup\big\{t\in \RRnn:\tau(\ideala^t)\not\subseteq \idealb\big\}
			=\min\big\{t\in \RRnn: \tau(\ideala^t)\subseteq \idealb\big\}.
	\end{equation*}

Critical Frobenius exponents are defined similarly, but with regular powers replaced with Frobenius powers.

\begin{defn}\label{defn: mu}
	For each $q$, we define  $\muab(q)=\max\{k\in \NN: \frob{\ideala}{k}\not\subseteq\frob{\idealb}{q}\}$.
\end{defn}

Because $\ideala\subseteq \sqrt\idealb$, we know that $\ideala^m\subseteq \idealb$ for some $m$.
If $\ideala$ is generated by $n$ elements, then 
	$\frob{\ideala}{mnq}\subseteq\ideala^{mnq}\subseteq \frob{\big(\ideala^{m}\big)}{q}\subseteq \frob{\idealb}{q}$,
	while $\frob{\ideala}{0}\not\subseteq \frob{\idealb}{q}$,
	so the maximum $k$  in the above definition always exists.	
Because $\frob{\ideala}{k}\not\subseteq\frob{\idealb}{q}$ if and only if $\frob{\ideala}{pk}\not\subseteq\frob{\idealb}{pq}$, 
	we have $\muab(pq)\ge p\cdot \muab(q)$, and thus,  
	as $q$ varies, $\seq{\muab(q)/q}$ forms a non-decreasing sequence, bounded above by $mn$.

\begin{defn}
	The \emph{critical Frobenius exponent of $\ideala$ with respect to $\idealb$}  is  
		\begin{equation*}
			\critb(\ideala)=\lim_{q\to\infty}\frac{\muab(q)}{q}=\sup_q \frac{\muab(q)}{q}.
		\end{equation*}
	We shall often omit the word ``Frobenius,'' and refer to $\critb(\ideala)$ simply as the 
		\emph{critical exponent of $\ideala$ with respect to $\idealb$}.
	We adopt the convention that $\crit^R(\ideala)$ is zero.
\end{defn}

\begin{notation}
	If $\ideala=\ideal{f}$, we denote $\ftb(\ideala)$ and $\critb(\ideala)$ by  $\ftb(f)$ and $\critb(f)$.
\end{notation}
	
We gather in the next proposition some properties of critical Frobenius exponents, analogous to those of $F$-thresholds.

\begin{prop}[Basic properties of critical Frobenius exponents]\label{prop: basic properties of critical exponents}\ 
	\begin{enumerate}[(1)]
		\item \label{item: basic inequalities for crit exps}
			If $f$ is a nonzero element of $\ideala$, then \[0<\ftb(f)=\critb(f)\le \critb(\ideala)\le \ftb(\ideala) < \infty.\]
		\item\label{item: crit exp in terms of integral powers}
		$\critb(\ideala)=\sup\big\{\tfrac{i}{q}\in \infint:\frob{\ideala}{i}\not\subseteq \frob{\idealb}{q}\big\}
		=\inf\big\{\tfrac{i}{q}\in \infint:\frob{\ideala}{i}\subseteq \frob{\idealb}{q}\big\}.$
		\item\label{item: crit exp in terms of real powers}
			$\begin{aligned}[t]
				\critb(\ideala)&=\sup\big\{t\in \RRnn: \frob{\ideala}{t}\not\subseteq \idealb\big\}
				=\min\big\{t\in \RRnn: \frob{\ideala}{t}\subseteq \idealb\big\}.
			\end{aligned}$
		\item\label{item: mu related to crit exp} 
			$\muab(q)<q \critb(\ideala)\le \muab(q)+1$, so that
				$\muab(q)=\up{q\critb(\ideala)}-1$.				
	\end{enumerate}
\end{prop}

\begin{proof}   
The discussion following \Cref{defn: mu} implies that $\critab$ is finite.  Next, observe that since $\idealb$ is proper,  $\bigcap_{q}\frob{\idealb}{q}\subseteq\bigcap_{q}\idealb^q=\ideal{0}$ by the Krull Intersection Theorem.  
Therefore, $f \notin \frob{\idealb}{q}$ for some $q$, and so 
	$\mu^\idealb_f(q) \geq 1$, which shows that $\critb(f)$ is at least $1/q$.  
The remaining inequalities in \iref{item:  basic inequalities for crit exps} follow from the containments $\ideal{f}^i=\frob{\ideal{f}}{i}\subseteq \frob{\ideala}{i}\subseteq \ideala^i$, which hold for every $i$.
	
	Next, consider the following sets:
		\begin{align*} 
			A&=\big\{\tfrac{i}{q}\in \infint:\frob{\ideala}{i}\subseteq \frob{\idealb}{q}\big\}&
			\bar{A}&=\big\{t\in \RRnn: \frob{\ideala}{t}\subseteq \idealb\big\}\\
			B&=\big\{\tfrac{i}{q}\in \infint:\frob{\ideala}{i}\not\subseteq \frob{\idealb}{q}\big\}&
			\bar{B}&=\big\{t\in \RRnn: \frob{\ideala}{t}\not\subseteq \idealb\big\}
		\end{align*}
	Because $\frob{\ideala}{i}\subseteq \frob{\idealb}{q}$ is equivalent to $\frob{\ideala}{i/q}\subseteq \idealb$, we have
		$A\subseteq \bar{A}$ and $B \subseteq \bar{B}$, hence	
		$\inf \bar{A}\le \inf A$ and $\sup B \le \sup\bar{B}$.
	Since there are sequences $\seq{\muab(q)/q}$ and $\seq{(\muab(q)+1)/q}$ in $B$ and in $A$ converging to $\critb(\ideala)$, we have
		\[\inf \bar{A}\le \inf A\le \critb(\ideala)\le \sup B\le \sup \bar{B}.\]
	By monotonicity, every element of $\bar{B}$ is less than every element of $\bar{A}$, so  
		$\sup \bar{B}\le \inf \bar{A}$, and  all quantities above are equal.
	Due to right constancy,  $\inf \bar{A}$ is actually a minimum. 
	This completes the proof of  \iref{item: crit exp in terms of integral powers} and \iref{item: crit exp in terms of real powers}.
	The inequalities in \iref{item: mu related to crit exp} then  follow from the fact that $\muab(q)/q\in B$, 
		while $(\muab(q)+1)/q\in A$.
\end{proof}

\begin{rem}
   \Cref{prop: basic properties of critical exponents}\iref{item: mu related to crit exp} shows, in particular, that every term $\muab(q)$ can be recovered from $\critab$.
   Moreover, if $q=p^e$, then \Cref{prop: basic properties of critical exponents}\iref{item: mu related to crit exp} is equivalent to saying that  
	$\muab(q)/q$ is the $e^\textrm{th}$ truncation of the unique nonterminating base $p$ expansion of $\critb(\ideala)$.  
\end{rem}

\begin{rem}
	If $\ideala$ is not a principal ideal, then only the first inequality in  \Cref{prop: basic properties of critical exponents}\iref{item: mu related to crit exp} need be true
		if we replace $\critb(\ideala)$ and $\muab(q)$ with $\ftb(\ideala)$ and $\nuab(q)$ 
		(see \cite[Proposition~1.7(5)]{mustata+takagi+watanabe.F-thresholds}). 
	However, \cite[Lemma~3.3]{stefani+betancourt+perez.existence-of-FTs} shows that if $\ideala$ can be generated 
		by $m$ elements, then
		\[\frac{\nuab(qq')}{qq'}\le \frac{\nuab(q)+m}{q}\]
		for each $q$ and $q'$.
	Fixing $q$ and letting $q'\to \infty$, we see that 
		\[\frac{\nuab(q)}{q} < \ftb(\ideala)\le  \frac{\nuab(q)+m}{q}.\]
\end{rem}

We conclude this subsection with the following results,  which state that critical exponents may be computed locally
	and show that critical exponents are invariant under split extensions.

\begin{prop}\label{prop:  crit can be computed locally}
	The critical exponent $\crit^{\idealb}(\ideala)$ is the maximum value among all $\crit^{\idealb_{\idealp}}(\ideala_{\idealp})$, 
	where $\idealp$ ranges over the prime ideals of $R$. 
\end{prop}

\begin{proof}
    If $\lambda = \crit^{\idealb}(\ideala)$, then $\frob{\ideala}{\lambda} \subseteq \idealb$, and  since localization commutes 
	    with Frobenius powers, $\frob{\ideala_{\idealp}}{\lambda} = (\frob{\ideala}{\lambda})_{\idealp} \subseteq \idealb_{\idealp}$.
	Therefore, $\crit^{\idealb_{\idealp}}(\ideala_{\idealp}) \leq \lambda$ for every $\idealp$.
	
 	To show that there exists $\idealp$ for which $\crit^{\idealb_{\idealp}}(\ideala_{\idealp}) = \lambda$, we follow the general
	argument in \cite[Lemma~2.13]{blickle+mustata+smith.F-thresholds_hyper}.  Take an ascending sequence $\seq{t_k}$ with limit $\lambda$. 
	By monotonicity, $\frob{\ideala}{t_{k+1}} \subseteq \frob{\ideala}{t_k}$ for every $k$.
	Thus, the ideals $(\idealb : \frob{\ideala}{t_k})$ form an ascending chain, and all are proper by definition of $\lambda$.
	Fix a prime ideal $\idealp$ containing their stabilization.    
	Since Frobenius powers and taking colons both commute with localization, it follows that  
	$(\idealb : \frob{\ideala}{t_k})_{\idealp} = (\idealb_{\idealp} : \frob{\ideala_{\idealp}}{t_k})$ 
	is a proper ideal of $ R_{\idealp}$ for all $k$, which allows us to conclude that $\lambda= \crit^{\idealb_{\idealp}}(\ideala_{\idealp})$. 
\end{proof}

\begin{prop}\label{prop: crit and direct summands}
	If $R \subseteq S$ is a split inclusion of $F$-finite regular domains, then $\crit^{\idealb}(\ideala) = \crit^{\idealb S}(\ideala S)$. 
\end{prop}

\begin{proof}
	Fix an $R$-linear map $\phi: S \to R$ that restricts to the identity on $R$.  Given a nonnegative integer $i$ and $q$ an integral 
		power of $p$, it suffices to show that $\frob{\ideala}{i}$ is contained in $\frob{\idealb}{q}$ if and only if 
		$\frob{\ideala}{i} S = \frob{(\ideala S)}{i}$ is contained in $\frob{\idealb}{q} S = \frob{(\idealb S)}{q}$.  However, if 
		$\frob{\ideala}{i}S \subseteq \frob{\idealb}{q}S$, then $\frob{\ideala}{i} = \phi(\frob{\ideala}{i} S) \subseteq \phi(\frob{\idealb}{q}S)  = \frob{\idealb}{q}$, 
		while the reverse implication is obvious.
\end{proof}

\subsection{Comparison with $F$-thresholds}

Here, we compare critical Frobenius exponents with $F$-thresholds, beginning with the terms of the sequences whose limits define them. 

\begin{prop}
	If  $\ideala\subseteq \idealb$, then $\muab(p)=\min\{\nuab(p),p-1\}$.
\end{prop}

\begin{proof}
	As $\ideala\subseteq \idealb$, we have $\frob{\ideala}{p}\subseteq \frob{\idealb}{p}$, so $\muab(p)\le p-1$.
	If $\nuab(p)\ge p-1$, then $\frob{\ideala}{p-1}=\ideala^{p-1}\not\subseteq \frob{\idealb}{p}$, and consequently 
		$\muab(p)=p-1$.
	If $\nuab(p)\le p-1$, then $\muab(p)=\nuab(p)$, since $\frob{\ideala}{k}=\ideala^k$ whenever $k<p$.
\end{proof}

\begin{prop}\label{uniform nu and mu comparison: P}
	Suppose that $\ideala \subseteq \idealb$.  
	If $\muab(q) \neq q-1$, then 
		\[   \muab(q) \geq \nuab(q) - \frac{(m-1)(q-1)}{p-1} \]
	whenever $\ideala$ can be generated by $m$ elements.
\end{prop}

\begin{proof}  
	Set $\mu = \muab(q)$ and $\nu = \nuab(q)$, and $l =  \frac{(m-1)(q-1)}{p-1}$.  
	If $\mu \neq q-1$, then $\mu+1 \leq q-1$, and \Cref{uniform comparison: L}  implies that
	$\ideala^{ l + \mu + 1}  = \ideala^{l}   \frob{\ideala}{\mu+1} \subseteq \frob{\idealb}{q}$,
	which allows us to conclude that $\nu \leq \mu +l$.
\end{proof}

\begin{cor}\label{uniform nu and mu comparison: C}
Suppose $\ideala \subseteq \idealb$.  If $\critb(\ideala) \neq 1$, then \[   \critab \geq  \ftab - \frac{m-1}{p-1}  \]
whenever $\ideala$ can be generated by $m$ elements.
\end{cor}

\begin{proof} 
	Our hypothesis that $\critb(\ideala) \neq 1$ implies that $\muab(q) \neq q-1$ for all $q \gg 0$, and our claim then follows 
		from \Cref{uniform nu and mu comparison: P}.  
\end{proof}

Recall that an $F$-jumping exponent for $\ideala$ is a positive number $\lambda$ such that $\tau(\ideala^{\lambda-\epsilon})\ne \tau(\ideala^\lambda)$, 
	for all  $0< \epsilon\le \lambda$ \cite[Definition~2.17]{blickle+mustata+smith.discr_rat_FPTs}.
In  \cite[Corollary~2.30]{blickle+mustata+smith.discr_rat_FPTs}, it was shown that the $F$-thresholds of $\ideala$ are precisely the $F$-jumping
	exponents of $\ideala$.
An analogous result holds in our setting.	

\begin{defn}
	$\Crit(\ideala)$ is the set consisting of all critical Frobenius exponents $\critb(\ideala)$,
		where $\idealb$ ranges over all proper ideals of $R$ with $\ideala \subseteq \sqrt{\idealb}$.
\end{defn}
	
\begin{prop}
	The set $\Crit(\ideala)$ of critical Frobenius exponents of $\ideala$ consists precisely of the jumping exponents 
		for the Frobenius powers of $\ideala$, that is, the positive real numbers $\lambda$ such that 
		$\frob{\ideala}{\lambda-\epsilon}\ne \frob{\ideala}{\lambda}$, for all $0<\epsilon\le \lambda$.
\end{prop}

\begin{proof}
	The characterization of critical exponents given in \Cref{prop: basic properties of critical exponents}\iref{item: crit exp in terms of real powers}
		shows that critical exponents are  jumping exponents. 
	The same characterization gives us the reverse containment: if $\frob{\ideala}{\lambda-\epsilon}\ne \frob{\ideala}{\lambda}$, 
		for all $0<\epsilon\le \lambda$, then 
		$\lambda=\min\{t\in \RRnn: \frob{\ideala}{t}\subseteq \frob{\ideala}{\lambda}\}=\crit^{\frob{\ideala}{\lambda}}\!(\ideala)$.
\end{proof}

\subsection{The least critical exponent}

Recall that the \emph{$F$-pure threshold} of a proper nonzero ideal $\ideala$ 
\cite{takagi+watanabe.F-pure_thresholds} is 
\[ \fpt(\ideala) = \sup \{  t  \in \RRpos : \tau(\ideala^{t}) = R \} = \min \{ t \in \RRpos : \tau(\ideala^{t}) \neq R \}, \]
which is well-defined since $\tau(\ideala^0) = R$, but $\tau(\ideala^{t}) \neq R$ for $t \gg 0$.  We adopt the convention that the $F$-pure threshold of the ideal $\ideala = R$ is infinite. Inspired by this, and by the observation that Frobenius powers satisfy the same properties that guaranteed that the above is well defined, analogously, we define:
	
\begin{defn}
	The \emph{least critical exponent} of a proper nonzero ideal $\ideala$ is 
	\[ \lce(\ideala) = \sup \{  t \in \RRpos : \frob{\ideala}{t} = R \} = \min \{ t \in\RRpos : \frob{\ideala}{t} \neq R \}. \]
	We adopt the convention that the least critical exponent of $\ideala=R$ is infinite.
\end{defn}

\begin{notation}
	If $\ideala=\ideal{f}$, we denote $\fpt(\ideala)$ and $\lce(\ideala)$ simply by $\fpt(f)$ and $\lce(f)$.
\end{notation}

The least critical exponent of $\ideala$ is, in fact, a critical exponent.  
Indeed, it follows from the definition that $\lce(\ideala) = \crit^{\idealb}(\ideala)$ for some proper ideal 
	$\idealb$ of $R$ if and only if $\idealb$ contains $\frob{\ideala}{\lce(\ideala)}$, which itself is proper by definition.  

\begin{defn}\label{defn: realizes}
If $\lce(\ideala) = \crit^{\idealb}(\ideala)$, then we say that $\idealb$ \emph{realizes} $\lce(\ideala)$.  
\end{defn}

A priori (that is, without computing the largest proper Frobenius power of $\ideala$), it is not at all clear which ideals realize $\lce(\ideala)$.  Below, we highlight two important cases in which such a determination is possible.

If $\ideala$ is a homogeneous and proper ideal of a polynomial ring over an $F$-finite field, then all 
	Frobenius powers of $\ideala$ are also homogeneous  (see, \eg \Cref{prop: roots in terms of generators}).  
Hence, in this case the unique homogeneous maximal ideal realizes $\lce(\ideala)$.  
Similarly, if $R$ is local, then the unique maximal ideal of $R$ realizes $\lce(\ideala)$ 
	for every nonzero proper ideal $\ideala$.

\begin{prop}\label{prop:  lce can be computed locally}
If $\ideala$ is a nonzero proper ideal of $R$, then $\lce(\ideala) \leq \lce(\ideala_{\idealp})$ for every 
	prime ideal $\idealp$ of $R$, with equality if and only if $\idealp$ contains the 
	proper ideal $\frob{\ideala}{\lce(\ideala)}$. 
\end{prop}

\begin{proof}
 If $\lambda = \lce(\ideala)$, then $\frob{\ideala}{t} = R$ for each $t<\lambda$, and since localization commutes with Frobenius powers,  
	 $\frob{\ideala_{\idealp}}{t} = R_{\idealp}$, for all such $t$, which demonstrates that $\lce(\ideala_{\idealp}) \geq \lambda$.     
	 On the other hand, if $\frob{\ideala}{\lambda} \subseteq \idealp$, then $\frob{\ideala}{\lambda}_{\idealp}$ is proper, which implies that $\lce(\ideala_{\idealp}) = \lambda$.  Otherwise, $\frob{\ideala}{\lambda}_{\idealp} = R_{\idealp}$, and the right constancy of Frobenius powers then implies that  $\lce(\ideala_{\idealp}) > \lambda$.
\end{proof}

\begin{prop} 
	Suppose that $\ideala$ is a nonzero proper ideal of $R$.
	\begin{enumerate}[(1)]
		\item	If $f$ is a nonzero element of  $\ideala$, then 
				\[0<\fpt(f)=\lce(f)\le \lce(\ideala)\le \min\{1,\fpt(\ideala)\}.\]
		\item If $\lce(\ideala) \neq 1$ and $\ideala$ can be generated by $m$ elements, then 
			\[  \lce(\ideala) \geq \fpt(\ideala) - \frac{m-1}{p-1}.\]
	\end{enumerate}
\end{prop}

\begin{proof}
The first claim follows from the fact that $\tau(f^{t}) = \frob{\ideal{f}}{t} \subseteq \frob{\ideala}{t} \subseteq \tau(\ideala^{t})$ for every $t > 0$, and that $\frob{\ideala}{1} = \ideala$, which is assumed to be proper.  

We pause to recall some basic properties of $F$-pure thresholds. First,  \Cref{prop: lce can be computed locally} holds for $F$-pure thresholds, after replacing $\lce(\ideala)$ with $\fpt(\ideala)$, and $\frob{\ideala}{\lce(\ideala)}$ with $\tau(\ideala^{\fpt(\ideala)})$.  Secondly, if $(R, \idealm)$ is local, then $\fpt(\ideala) = \ft^{\idealm}(\ideala)$.  

We now address the second claim.  If $\lce(\ideala) \neq 1$, then there exists a prime ideal $\idealp$ of $R$ for which $\lce(\ideala_{\idealp}) = \lce(\ideala)  \neq 1$, and so
\[ \lce(\ideala) = \lce(\ideala_{\idealp}) \geq \fpt(\ideala_{\idealp}) - \frac{m-1}{p-1} \geq \fpt(\ideala) - \frac{m-1}{p-1}, \] 
where the first inequality follows from  \Cref{uniform nu and mu comparison: C}, taking $\idealb$ to be the maximal ideal of $R_\idealp$, and the second from the analog of \Cref{prop: lce can be computed locally} for $F$-pure thresholds.
\end{proof}

We conclude this subsection with some examples, contrasting the different behavior of least critical exponents and $F$-pure thresholds.

\begin{exmp}  
	The formulas discussed in \Cref{subs:  Examples} lead to the following well-known description:  
	The $F$-pure threshold of a monomial ideal is the unique real number  $\lambda$ such that $(1/\lambda,\ldots,1/\lambda)$ lies in 
		the boundary of its Newton polyhedron.  
	In particular, 
		\[ \fpt\big( \ideal{ x_1, \ldots, x_n }^d \big) = \frac{n}{d}.\]  
	The general situation for the least critical exponent of a monomial ideal is, however, rather complex.  
	For instance, \Cref{exmp: Comparison with integral closure}  tells us 
		that if $\idealb = \ideal{x, y }^5$, then $\fpt(\idealb) = 2/5$, and 
		\[ \lce(\idealb) = \frac25 - \frac{1}{5p} \] whenever $p \equiv 3 \bmod 5$. 
	More generally, we have computed that
		\[
		\lce(\idealb)=
		\begin{cases} 
			 \frac25&\text{if } p = 5  \text{ or }  p \equiv \pm 1 \bmod 5 \\
			 \frac25 - \frac{1}{5p^3} &\text{if } p \equiv 2 \bmod 5	\\
			 \frac25 - \frac{1}{5p} &\text{if } p \equiv 3 \bmod 5	
		\end{cases}
		\]
	We stress that this example illustrates that for an ideal $\ideala$, it is possible, and perhaps even common, to have 
		$\lce(\ideala) < \fpt(\ideala)$ when $\fpt(\ideala) < 1$.  
\end{exmp}

\section{The Principal Principle} \label{s: principal principle}

As noted in \Cref{prop: basic properties of real powers}\iref{item: real power vs test ideal}, the Frobenius powers 
	and test ideals of a principal ideal in an $F$-finite regular domain agree for all positive exponents.  
On the other hand, we also observed in \Cref{cor: skoda} and in the comments preceding it, a situation in which 
	the Frobenius powers of an \emph{arbitrary} ideal behave like the test ideals of principal ideals.  

In this section, we make explicit the connection between test ideals of hypersurfaces and Frobenius powers of arbitrary ideals, and explore some consequences. 
	The results derived in this section suggest that the following heuristic principle can be used when dealing with Frobenius powers:

\begin{pple}
	Given a result for test ideals or $F$-thresholds or $F$-pure thresholds
	of principal ideals, there is  an analogous result for Frobenius powers or critical exponents or least critical exponents.
\end{pple}

\subsection{Principalization}
	
\begin{prop}\label{prop: generic linear combinations 1}  
	Fix an ideal $\ideala = \ideal{g_1, \ldots, g_m}$ of an $F$-finite regular domain $R$ of characteristic $p>0$.  
	Let  $\vz = z_1, \ldots, z_m$ be variables over $R$, and consider the generic linear combination  
		$G = z_1 g_1 + \cdots + z_m g_m \in R[\vz]$.  
	If $\idealb$ is an ideal of $R$ and $t$ is a positive real number, then 
		\[ \frob{\ideala}{t} \subseteq \idealb \iff \tau(G^t) \subseteq \idealb R[\vz].\]   
	Consequently,  if $\ideala$ and $\idealb$ are nonzero proper ideals of $R$ with $\ideala \subseteq \sqrt{\idealb}$, 
		then \[\critb(\ideala) = \ft^{\idealb R[\vz]}(G).\]	
\end{prop}

\begin{proof}
	We may assume that $t=k/q$, with $k$ a positive integer and $q$ a power of $p$.  
	With such a choice of parameter, our claim then is equivalent to the assertion that 
		$\frob{\ideala}{k} \subseteq \frob{\idealb}{q}$ if and only if $G^k \in \frob{(\idealb R[\vz])}{q} = \frob{\idealb}{q} R[\vz]$.  
	However, $G^k \in \frob{\idealb}{q} R[\vz]$ if and only if each of the coefficients in the expression of $G^k$ as an 
		$R$-linear combination of monomials in the variables $\vz$ lies in $\frob{\idealb}{q}$, and 
		\Cref{prop: generators for Frobenius powers} 
		tells us that these coefficients are precisely the generators of $\frob{\ideala}{k}$.
\end{proof}

\begin{rem}
	Setting $\idealb = \frob{\ideala}{t}$ in \Cref{prop: generic linear combinations 1} implies that 
		$\tau(G^t)$ is contained in $\frob{\ideala}{t} R[\vz]$ for all $t >0$.
	Though this containment may be proper (\eg when $t=1$), we will see in \Cref{thm: principalization} 	
		that we obtain an equality whenever $0 < t < 1$.
\end{rem}

In the proof of the following theorem, we refer to \cite[Proposition~2.5]{blickle+mustata+smith.discr_rat_FPTs}.
This proposition, which allows us to compute Frobenius roots of ideals in terms of their generators, is restated in \Cref{s:  algorithm}.

\begin{thm}[Principalization]\label{thm: principalization}  
	Suppose $0 < t < 1$.  
	If $\ideala \subseteq R$ and $G \in R[\vz]$ are as in \Cref{prop:  generic linear combinations 1}, then 
		$\tau(G^t)=\frob{\ideala}{t}R[\vz]$ and $\frob{\ideala}{t}=\tau(G^t)\cap R$.
\end{thm}

\begin{proof}   
	Note that the second statement follows from the first. 
	Indeed, as $R[\vz]$ is split over $R$, we have that $(\idealb R[\vz]) \cap R = \idealb$ for every ideal $\idealb$ of $R$.

	We now turn our attention to the first statement.  
	First, note that we may assume that $t=k/q$, with $0<k<q$.  
	Moreover, it suffices to establish that the $R$-modules $\frob{\ideala}{t} R[\vz]$ and $\tau(G^t)$ are 
		equal after localizing at each prime ideal of~$R$.  
	However, if $\idealp$ is a prime ideal of $R$, then under the identification $(R[\vz])_\idealp = R_\idealp[\vz]$, 
		the localization of the expansion $\frob{\ideala}{t} R[\vz]$ at $\idealp$ satisfies
			\[\big(\frob{\ideala}{t} R[\vz]\big)_{\!\idealp} = \big(\frob{\ideala}{t}\big)_{\!\idealp} R_\idealp[\vz] = \frob{\ideala_\idealp}{t} R_\idealp[\vz],\]
		while the localization of $\tau(G^t)$ at $\idealp$ is identified with the test ideal of $G$, 
		regarded as an element of $R_\idealp[\vz]$, with respect to the parameter $t$.  
	In other words, we may assume that $R$ is local, and therefore free over its subring $R^q$.
		  
  	If  $B=\{e_1,\ldots,e_n\}$ is a basis for $R$ over $R^q$, then 
		$B'=\{e_i \vz^{\vv{u}}: 1\le i\le n\text{ and } \vv{u}<q\vv{1}\}$ is a basis for $R[\vz]$ over $(R[\vz])^q$.	
	Let $\mathscr{H} $ denote the collection of all vectors $\vv{h}\in \NN^m$ such that $\norm{\vv{h}}=k$ 
		and $\binom{k}{\vv{h}}\not\equiv 0\bmod p$.
	By \Cref{prop: generators for Frobenius powers}, $\frob{\ideala}{k}$ is generated  by products 
		$g^\vv{h}=g_1^{h_1}\cdots g_m^{h_m}$, where $\vv{h}\in \mathscr{H}$.
		Writing each $g^\vv{h}$ in terms of the basis $B$ as 
		\[g^\vv{h}=\sum_{i=1}^n a_{\vv{h},i}^q e_i,\]
		where each $a_{\vv{h},i}\in R$,
	 	\cite[Proposition~2.5]{blickle+mustata+smith.discr_rat_FPTs} shows that $\frob{\ideala}{k/q}$ is generated by 	
			\begin{equation}\label{eqn: generating set}
			\big\{a_{\vv{h},i}:\vv{h}\in \mathscr{H}, 1\le i\le n \big\}.
		\end{equation}
	On the other hand, we have
		\[G^k=\sum_{\vv{h}\in \mathscr{H}} \binom{k}{\vv{h}} g^\vv{h} \vz^\vv{h}
			= \sum_{\substack{\vv{h}\in \mathscr{H} \\1\le i\le n}} \binom{k}{\vv{h}} a_{\vv{h},i}^q e_i \vz^\vv{h}.\]
	Since each $\vv{h}\in \mathscr{H}$ has norm $k<q$, the above is an expression for $G^k$ on the basis $B'$, so \cite[Proposition~2.5]{blickle+mustata+smith.discr_rat_FPTs} shows that $\frob{\ideal{G^k}}{1/q}$, which equals $\tau(G^{k/q})$ by \cite[Lemma~2.1]{blickle+mustata+smith.F-thresholds_hyper}, is also generated (in $R[\vz]$) by the set displayed in \eqref{eqn: generating set}.
\end{proof}

\begin{rem}
	Together with  \Cref{cor: skoda}, \Cref{thm:  principalization} allows us to express any real Frobenius power as 	
		the product of an integral Frobenius power and the test ideal 
		of a principal ideal with respect to a parameter in the open unit interval.  
	This has an important computational consequence: Frobenius powers of arbitrary ideals can be 
		explicitly determined using algorithms for the computation of test ideals of principal ideals 
		(we refer the reader to \cite{hernandez+etal.local_m-adic_constancy} for such an algorithm).  
	This approach is typically efficient for ideals generated by a small number of elements.
	Nevertheless, we present an algorithm for the direct computation of 
		rational Frobenius powers in \Cref{s: algorithm}.
\end{rem}
	
\begin{cor}\label{cor: principalization for lce}
	If $\ideala =\ideal{g_1,\ldots,g_m}$ is a nonzero proper ideal of $R$ and 
		$G\in R[\vz]$ is as in \Cref{prop:  generic linear combinations 1}, then \[ \lce(\ideala)=\fpt(G).\] 
\end{cor}

\begin{proof}
	Let $\lambda=\fpt(G)$.
	Then $\tau(G^t)=R[\vz]$ whenever $0\le t<\lambda$, and  in view  of \Cref{thm: principalization}, 
		$\frob{\ideala}{t}=R$ for all such $t$, showing that $\lce(\ideala)\ge \lambda$.
	If $\lambda=1$, then it must be the case that $\lce(\ideala)=1=\fpt(G)$.
	If $\lambda<1$, then $\tau(G^\lambda)\ne R[\vz]$, and $\frob{\ideala}{\lambda}=\tau(G^\lambda)\cap R\ne R$, showing that $\lce(\ideala)=\lambda=\fpt(G)$.	 
\end{proof}

\begin{rem}[An alternate form of principalization] \label{rem: alternate principalization}
	Let  $\ideala \subseteq R$ and $G \in R[\vz]$ be as in \Cref{prop:  generic linear combinations 1}.  Let $\kk$ be an $F$-finite field contained in $R$, and 
		consider $A = R \otimes_{\kk} \kk(z)$, which we identify with the localization of 
		$R[\vz] = R \otimes_{\kk} \kk[\vz]$ at the set of nonzero polynomials in $\kk[\vz]$.  

	A useful fact in this context is that $(\idealb A)\cap R[\vz] = \idealb R[\vz]$, and hence that 
		$(\idealb A ) \cap R = \idealb$, for every ideal $\idealb$ of $R$.  
	Indeed, after fixing a monomial order on the variables $\vz$, one can show by induction on the number of 
		terms that every coefficient of every polynomial in $(\idealb A) \cap R[\vz]$ must lie in $\idealb$.  
	This demonstrates that $(\idealb A) \cap R[\vz] \subseteq \idealb R[\vz]$, while the reverse containment holds trivially. 

	Now, let $\tau(G^t)$ be the test ideal of $G \in R[\vz]$, and $\tau_A(G^t)$ the test ideal of $G$, 
		when regarded as an element of the localization $A$.
	As test ideals commute with localization, we have that $\tau_A(G^t) = \tau(G^t) A$, and combining this 
		with the equality $\tau(G^t) = \frob{\ideala}{t} R[\vz]$ from \Cref{thm: principalization} shows that 
		$\tau_A(G^t) =  \frob{\ideala}{t} A$.  
	Furthermore, intersecting the above with $R$ shows that
		\[ \tau_A(G^t) \cap R = (\frob{\ideala}{t} A ) \cap R = \frob{\ideala}{t}.\]

	In summary, we have just seen that an analog of  \Cref{thm: principalization} holds for the generic linear 
		combination $G$, regarded as an element of $A$.  
	This observation also leads to the following analog of \Cref{cor: principalization for lce}:  
		If $\ideala$ is a nonzero proper ideal, then the least critical exponent of $\ideala$ equals the $F$-pure threshold of $G$, regarded as an element of $A$.
\end{rem}

We now record an immediate consequence of \Cref{rem: alternate principalization}.

\begin{cor} \label{cor: alternate principalization}
	Fix an ideal $\ideala = \ideal{g_1, \ldots, g_m} \subseteq \kk[x_1, \ldots, x_n] = \kk[\vx]$, 
		with $\kk$ an $F$-finite field of characteristic $p>0$.  
	Let  $\vz = z_1, \ldots, z_m$ be variables over $\kk[\vx]$, and consider the generic linear combination 
		\[ G = z_1 g_1 + \cdots + z_m g_m \in \kk(\vz)[\vx].\]   
	If $0 < t < 1$, then $\tau(G^t)=\frob{\ideala}{t}\cdot\kk(\vz)[\vx]$ and $\frob{\ideala}{t}=\tau(G^t)\cap \kk[\vx]$.  
	In particular, if $\ideala$ is a nonzero proper ideal, then
		\[
		\pushQED{\qed} 
		\fpt(G) = \lce(\ideala).\qedhere
		\popQED
		\]     
 \end{cor}

In the remainder of this subsection, we derive some straightforward consequences of the preceding results.  
First, we point out that since the $F$-thresholds of the polynomial 
	$G$ in \Cref{prop: generic linear combinations 1} are rational and form a discrete set
 	(see  \cite[Theorem~1.1]{blickle+mustata+smith.F-thresholds_hyper}), 
	the discreteness and rationality of the critical exponents of $\ideala$ follows at once.
	
\begin{cor}  
	If $\ideala$ is a nonzero proper ideal of an $F$-finite regular domain, then the set  $\Crit(\ideala)$ of critical exponents of 
		$\ideala$ is discrete and contained in $\QQ$.
	\qed
\end{cor}

Principalization also sheds light on the structure of least critical exponents, and highlights an important way 
	in which they must differ from $F$-thresholds for non-principal ideals.   
Recall that although every rational number is the $F$-pure threshold of some ideal, there are 
	``forbidden" intervals in $(0,1)$ containing no $F$-pure threshold of a principal ideal
	(see \cite[Proposition~4.3]{blickle+mustata+smith.F-thresholds_hyper} and the discussion that precedes it, and \cite[Proposition~4.8]{hernandez.F-purity_of_hypersurfaces}).  
Since least critical exponents are $F$-pure thresholds of principal ideals, they must avoid the same intervals.  
	
\begin{cor}[Forbidden interval condition]\label{cor: forbidden intervals} 
	The least critical exponent of any nonzero proper ideal of an $F$-finite regular domain 
		does not lie in any interval of the form 
		$\big(\frac{k}{q},\frac{k}{q-1}\big)$, where $q$ is an integral power of $p$, and $0\le k\le q-1$.
	\qed
\end{cor}

Suppose that $\idealb$ realizes $\lce(\ideala)$, in the sense that $\lce(\ideala) = \crit^{\idealb}(\ideala)$.  
In view of the above corollary, \Cref{prop: basic properties of critical exponents}\iref{item: mu related to crit exp} 
	can be strengthened in this situation to
		\begin{equation}\label{eq: nu bounds for lce}
			\frac{\muab(q)}{q-1}\le \lce(\ideala) \le \frac{\muab(q)+1}{q}.
		\end{equation}
This gives us the following characterization of ideals with least critical exponent equal to $1$.
	
\begin{cor}
	 If $\ideala$ is a nonzero proper ideal of an $F$-finite regular domain, and  $\idealb$ is such that $\lce(\ideala) = \crit^{\idealb}(\ideala)$, then the following are equivalent\textup:
	\begin{enumerate}[(1)]
		\item $\lce(\ideala)=1$.
		\item $\muab(q)=q-1$ for  all  $q$.
		\item $\muab(q)=q-1$ for  some  $q$.
	\end{enumerate}
\end{cor}

\begin{proof}	
	If $\lce(\ideala)=1$, then \Cref{prop: basic properties of critical exponents}\iref{item: mu related to crit exp}
		shows that $\muab(q)=q-1$ for  all  $q$.
	If $\muab(q)=q-1$ for \emph{some} $q$, \eqref{eq: nu bounds for lce}  shows that
			$\lce(\ideala)\ge 1$, so $\lce(\ideala)=1$. 
\end{proof}

Next, we turn our attention to the subadditivity property.  
In the context of $F$-thresholds, subadditivity says that if $\ideala, \idealb, \mathfrak{d}$ are
	nonzero proper ideals of an 
	$F$-finite regular ring with $\ideala, \idealb \subseteq \sqrt{\mathfrak{d}}$, then 
	$\ft^{\mathfrak{d}}(\ideala+\idealb) \leq \ft^{\mathfrak{d}}(\ideala) + \ft^{\mathfrak{d}}(\idealb)$ 
	\cite[Lemma~3.3]{blickle+mustata+smith.F-thresholds_hyper}.  
The proof of subadditivity for $F$-thresholds is simple, and is based on the observation that if 
	$\ideala^m$ and $\idealb^n$ are contained in $\frob{\mathfrak{d}}{q}$, then so is $(\ideala + \idealb)^{m+n}$.  
However, although it is not immediately clear that the same observation holds after replacing regular powers with 
	Frobenius powers, subadditivity holds for critical exponents.

\begin{cor} \label{lem: bound on sum of crits}
	If $\ideala, \idealb, \mathfrak{d}$ are nonzero proper ideals of an $F$-finite regular domain $R$  with $\ideala, \idealb \subseteq \sqrt{\mathfrak{d}}$, then
		$\crit^\mathfrak{d}(\ideala+\idealb)\le \crit^\mathfrak{d}(\ideala)+\crit^\mathfrak{d}(\idealb)$.	
\end{cor}	

\begin{proof}
	Suppose that $\ideala = \ideal{g_1, \ldots, g_m}$ and $\idealb = \ideal{h_1, \ldots, h_n}$, 
		and fix variables $\vz = z_1, \ldots, z_m$ and $\vw = w_1, \ldots, w_n$ over $R$.  
	By \Cref{prop: generic linear combinations 1,prop: crit and direct summands}, there exist $G \in R[\vz]$ 
		and $H \in R[\vw]$ satisfying the following conditions:
		\begin{itemize}
			\item $\crit^{\mathfrak{d}}(\ideala) = \ft^{\mathfrak{d} R[\vz]}(G) = \ft^{\mathfrak{d} R[\vz,\vw]}(G)$, 
			\item $\crit^{\mathfrak{d}}(\idealb) = \ft^{\mathfrak{d} R[\vw]}(H) = \ft^{\mathfrak{d} R[\vz,\vw]}(H)$, and 
			\item $\crit^{\mathfrak{d}}(\ideala+\idealb) = \ft^{\mathfrak{d} R[\vz,\vw]}(G+H)$.
		\end{itemize}
	The claim then follows from subadditivity for $F$-thresholds in $R[\vz,\vw]$.	
\end{proof}

\subsection{Principalization in a polynomial ring}

In this section, we specialize to the case of a polynomial ring over an $F$-finite field.
In this context, we prove a stronger version of the second claim in \Cref{prop: generic linear combinations 1} for critical exponents with respect to a monomial ideal.
Our argument relies on the following well-known result.

\begin{lem}\label{lem: discreteSetBMS}
	Fix positive integers $n$ and $d$.  
	Then the set consisting of all $F$-jumping exponents of all ideals of $\kk[x_1,\ldots,x_n]$ generated in degree at most $d$, 
		where $\kk$ ranges over all $F$-finite fields of characteristic $p$, is discrete.
\end{lem}
	
\begin{proof}
	According to \cite[Proposition~3.8, Remark~3.9]{blickle+mustata+smith.discr_rat_FPTs}, the set in question
		is contained in a union of sets of the form $\frac{1}{p^a(p^b-1)}\NN$, where $a,b$ are nonnegative 
		integers bounded above by constants that depend only on $n, d$, and $p$, but not on $\kk$.
	In particular, this set is contained in a \emph{finite} union of discrete sets.
\end{proof}

\begin{setup}\label{su:  stratification}
	Let $\kk$ be an $F$-finite field of characteristic $p>0$.  Let 
		\[ g_1, \ldots, g_m \in \kk[x] = \kk[x_1, \ldots, x_n] \] 
		be nonzero polynomials that generate a proper ideal $\ideala$, 
		and fix a proper monomial ideal $\idealb$ of $\kk[x]$ with $\ideala \subseteq \sqrt{\idealb}$. 

	If $\LL$ is an $F$-finite field containing $\kk$, and $\mathfrak{d}$ is an ideal of $\kk[x]$, 
		then $\mathfrak{d}_{\LL}$ is the extension of $\mathfrak{d}$ to $\LL[x]$.  
	With this notation, \Cref{prop: crit and direct summands} tells us that
		\[ \crit^{\idealb}(\ideala) = \crit^{\idealb_{\LL}}(\ideala_{\LL}).\]  
	Similarly, if $V$ is a closed set of $\mathbb{A}^m_{\kk}$, then $V_{\LL}$ is the base change of 
		$V$ to $\LL$, \ie the subset of $\mathbb{A}^m_{\LL}$ determined by the same equations that 
		define $V$ in $\mathbb{A}^m_{\kk}$.
\end{setup}

\begin{thm}\label{thm: generic linear combinations 2} 
	Under \Cref{su:  stratification}, there exists a closed set $V$ of $\mathbb{A}^m_{\kk}$ with the following property\textup:
	Given an $F$-finite field $\LL$ containing $\kk$, and an $m$-tuple 
		$\vvv{\gamma} = (\gamma_1 ,\ldots, \gamma_m) \in \mathbb{A}^m_{\LL}$ for which 
		$g=\gamma_1 g_1+ \cdots + \gamma_m g_m$ is nonzero,  
		\[ \ft^{\idealb_{\LL}}(g) = \critb(\ideala)  \iff \vvv{\gamma} \notin V_{\LL}.\]
\end{thm}

\begin{proof}
	Let $G=z_1g_1+\cdots  +z_m g_m\in  \kk[\vx,\vz]$  
		and let $\idealB$ be the extension of $\idealb$ in $\kk[\vx,\vz]$. 
	Furthermore, let  $\lambda$ denote the common value $\lambda =\ft^{\idealB}(G) = \critb(\ideala)$. 
	
	By \Cref{lem: discreteSetBMS}, the set consisting of all the $F$-jumping exponents of all nonzero linear combinations 
		of $g_1,\ldots,g_m$ with coefficients in any $F$-finite extension field of $\kk$ is discrete.  
	Thus, there exists an interval of the form $(\lambda-\epsilon,\lambda)$ disjoint from that set.
	Choose a $p$-rational number $i/q$ in that interval, so that $G^i\notin \frob{\idealB}{q}$. 
	Write $G^i$ in the form 
		\begin{equation}\label{eqn: H def}
			G^i = \sum_{\vv{u}}H_{\vv{u}}(\vz) \, \vx^{\vv{u}},
		\end{equation}
		with nonzero $H_{\vv{u}}\in \kk[\vz]$.
	Because $G^i\notin \frob{\idealB}{q}$,  there is at least one monomial $\vx^{\vv{u}}$ in \eqref{eqn: H def} 
		not in $\frob{\idealb}{q}$.   	
	Let $X$ be the collection of all such $\vv{u}$, set $\mathcal{H}=\{H_{\vv{u}}: \vv{u} \in X\}$, 
		and let $V$ be the closed subset of $\mathbb{A}^m_{\kk}$ defined by $\mathcal{H}$.

	Now, fix an $F$-finite extension field $\LL$ of $\kk$, and 
		$\vvv{\gamma}=(\gamma_1, \ldots, \gamma_m) \in \mathbb{A}^m_{\LL}$ for which 
		$g=\gamma_1 g_1+ \cdots + \gamma_m g_m$ is nonzero.  The identity \eqref{eqn: H def} implies that 
			\[ 	g^i \equiv \sum_{\vv{u} \in X}H_{\vv{u}}(\vvv{\gamma}) \, \vx^{\vv{u}}  \mod \frob{\idealb_{\LL}}{q} \]
		in $\LL[x]$.  
	Since $\idealb$ is monomial, this implies that $g^i \notin \frob{\idealb_{\LL}}{q}$ 
		if and only if $\vvv{\gamma} \notin V_{\LL}$.  

	It remains to show that $g^i \notin \frob{\idealb_{\LL}}{q}$ if and only if $\ft^{\idealb_{\LL}}(g) = \lambda$.  Towards this, first note that if $g^i \in \frob{\idealb_{\LL}}{q}$, then $\ft^{\idealb_{\LL}}(g)$ is at most $i/q$, 
		and hence less than $\lambda$.  
	On the other hand, if $g^i \notin \frob{\idealb_{\LL}}{q}$, then  $\ft^{\idealb_{\LL}}(g)\ge i/q$.  
	But $\ft^{\idealb_{\LL}}(g)$ cannot lie in the interval $(\lambda-\epsilon,\lambda)$, by design, so 
		$\ft^{\idealb_{\LL}}(g) \ge \lambda=\critb(\ideala) = \crit^{\idealb_{\LL}}(\ideala_{\LL})$.
	\Cref{prop: basic properties of critical exponents}\iref{item: basic inequalities for crit exps} 
		then tells us that equality holds throughout.
\end{proof}

Note that the closed set $V$ constructed in the proof of \Cref{thm: generic linear combinations 2} is defined by 
	a finite collection of nonzero homogeneous polynomials.  
	If $\kk$ is an infinite field, the fact that these polynomials are nonzero implies that $V \subseteq \mathbb{A}^m_{\kk}$ is proper, 
	which gives us the following.
	
\begin{cor}\label{cor: crit(a) = max(c(g): g in a)}
	If $\kk$ is infinite, then under the assumptions of \Cref{thm: generic linear combinations 2} we have
		$\critb(\ideala)=\max\{\ftb(g):g\in\ideala\}.$
	\qed
\end{cor}

\begin{exmp}
We examine the proof of \Cref{thm: generic linear combinations 2} in the concrete case that the polynomials $g_1, \ldots, g_m$ are monomials, say $g_i = \vx^{\vv{u}_i}$ for every $1 \leq i \leq m$.  In this case,  $G = z_1 \vx^{\vv{u}_1} + \cdots + z_m \vx^{\vv{u}_m}$, and so 
\[ G^i = \sum \binom{i}{\vv{h}} z^{\vv{h}} \vx^{h_1 \vv{u}_1 + \cdots + h_m \vv{u}_m}\] 
by the multinomial theorem, where the sum is over all $\vv{h}\in \NN^m$ with $\norm{\vv{h}}=i$ and $\binom{i}{\vv{h}}\not\equiv 0\bmod{p}$. Therefore, for each fixed monomial $x^{\vv{u}}$ appearing in the expression for $G^i$ given in \eqref{eqn: H def}, we have that
\[H_{\vv{u}}(z)=\sum\binom{i}{\vv{h}}\vz^{\vv{h}}, \]
		where the sum ranges over all 
		$\vv{h}\in \NN^m$ satisfying the preceding two conditions, as well as the additional condition 
\[ h_1 \vv{u}_1 + \cdots + h_m \vv{u}_m=\vv{u}.\] 

Our choice of the monomial $\vx^{\vv{u}}$ guarantees that there exists at least one $\vv{h}$ satisfying these conditions, and in certain cases, there is only a single vector  $\vv{h}$ that satisfies them.  
	In such cases, the polynomial $H_{\vv{u}}(z)$  is a monomial, and so the closed set that it defines in $\mathbb{A}^m_{\FF_p}$ lies in the union of the coordinate hyperplanes.  In particular, the proof of 	 \Cref{thm: generic linear combinations 2} tells us that \[ \ft^{\idealb_{\LL}}(g) = \crit^{\idealb}(\ideala) \]
		whenever $\LL$ is an $F$-finite field of characteristic $p>0$ and $g$ is any $\LL$-linear combination with nonzero coefficients of the monomials $\vx^{\vv{u}_1}, \ldots, \vx^{\vv{u}_m}$.  	
		
	This behavior occurs, for instance, when $\vv{u}_1, \ldots, \vv{u}_m$ are linearly independent (\eg this occurs for the monomials $x_1^{d_1},\ldots,x_m^{d_m}$).		
	Considering the condition that
		$\norm{\vv{h}}=i$, we see that this is also the case when these vectors are affinely independent  (\eg this occurs for a distinct pair of monomials 		
		$\vx^{\vv{u}_1}$ and $\vx^{\vv{u}_2}$).  
\end{exmp}

Techniques similar to those used in the proof of \Cref{thm: generic linear combinations 2} can be used to prove the result below, which is not directly related to generalized Frobenius powers, nor to critical exponents.
Although this result, and its corollary---the constructibility of sets of polynomials with a given $F$-pure threshold---and the arguments used toward them are known to specialists, the authors are unaware of proofs of these statements in the literature, and thus choose to present them in detail.

\begin{prop}\label{prop: stratification} 
Under \Cref{su:  stratification}, the following hold.
\begin{enumerate}[(1)]
\item Given an $F$-finite field $\LL$ containing $\kk$, let $S(\LL)$ denote the set of $F$-thresholds with respect to $\idealb_{\LL}$ of all nonzero $\LL$-linear combinations of $g_1, \ldots, g_m$.  Then  $S = \bigcup \, S(\LL)$ is finite, where the union ranges over all $\LL$ as above.
\item For each $\lambda \in S$, there exists a locally closed set $Z^{\lambda}$ of $\mathbb{A}^m_{\kk}$ with the following property\textup:  Given an $F$-finite field $\LL$ containing $\kk$, and an $m$-tuple $\vvv{\gamma} = (\gamma_1 ,\ldots, \gamma_m) \in \mathbb{A}^m_{\LL}$ for which $g=\gamma_1 g_1+ \cdots + \gamma_m g_m$ is nonzero,  
		 \[ \ft^{\idealb_{\LL}}(g) = \lambda  \iff \vvv{\gamma} \in Z^{\lambda}_{\LL}.\]
\end{enumerate}
\end{prop}

\begin{proof}
    Let $G=z_1g_1+\cdots  +z_m g_m$, and let $\idealB$ be the extension of $\idealb$ in $\kk[\vx,\vz]$.  
	The set $S$ is finite in view of  \Cref{lem: discreteSetBMS}, since given an $F$-finite extension field $\LL$ of $\kk$, 
	every element of $S(\LL)$ does not exceed $\ft^\idealB(G)$. 
	Therefore, we can list the elements of $S$ as  $\lambda_1<\lambda_2<\cdots<\lambda_r$.
	
	Given $1\le j< r$,  fix a $p$-rational number $i_j/q$ in the interval $(\lambda_j,\lambda_{j+1})$. 
	Since
	\[
	i_j/q < \lambda_r\le \ft^\idealB(G),
	\] 
	we know that $G^{i_j}\notin \frob{\idealB}{q}$.
	This means that $G^{i_j}$, when thought of as a polynomial in $\vx$ with coefficients in $\kk[\vz]$, has at least one supporting monomial not in $\frob{\idealb}{q}$.
	As in the proof of \Cref{thm: generic linear combinations 2}, we gather the coefficients of 
		all such monomials, and let  $V^j$ be the closed set in $ \mathbb{A}^m_{\kk}$ defined by these coefficients.
		Moreover, set $V^0=\emptyset$ and $V^r=\mathbb{A}^m_{\kk}$, so that $V_j$ is defined for each $0 \leq j \leq r$.
		
	Fix an  $F$-finite extension field $\LL$ of $\kk$, and take $\vvv{\gamma}=(\gamma_1, \ldots, \gamma_m) \in\mathbb{A}^m_{\LL}$ with
		$g=\gamma_1 g_1+ \cdots + \gamma_m g_m$ nonzero.
	As in the proof of \Cref{thm: generic linear combinations 2}, 
	for $1 \leq j < r$, $\vvv{\gamma} \in V^j_{\LL}$ if and only if $g^{i_j} \in \frob{\idealb_{\LL}}{q}$, and we claim that the latter condition is equivalent to the condition that $\ft^{\idealb_{\LL}}(g) \leq \lambda_j$, so that
	for each $1\le j<r$ we have
	\begin{equation} \label{eq: Vj equivalence}
	\vvv{\gamma} \in V^j_{\LL}  \iff \ft^{\idealb_{\LL}}(g) \leq \lambda_j.
	\end{equation}	
        Indeed, the condition that $g^{i_j} \in \frob{\idealb_{\LL}}{q}$ implies that $\ft^{\idealb_{\LL}}(g) \leq i_j/q < \lambda_{j+1}$, and since there are no elements of $S(\LL)$ between $\lambda_{j}$ and $\lambda_{j+1}$, we conclude that $\ft^{\idealb_{\LL}}(g) \leq \lambda_{j}$.
        On the other hand, if $g^{i_j} \notin \frob{\idealb_{\LL}}{q}$, then $\ft^{\idealb_{\LL}}(g) \geq i_j/q > \lambda_j$. 
		
The equivalence \eqref{eq: Vj equivalence} tells us that for every $1 < j < r$,
\begin{align*}
\ft^{\idealb_{\LL}}(g) = \lambda_j & \iff \ft^{\idealb_{\LL}}(g) \leq \lambda_j \text{ and } \ft^{\idealb_{\LL}}(g) > \lambda_{j-1} \\ 
& \iff \vvv{\gamma} \in V^j_{\LL} \text{ and }  \vvv{\gamma} \notin V^{j-1}_{\LL}.
\end{align*}
Finally, note that \eqref{eq: Vj equivalence} implies that the first and last conditions above are also equivalent for $j=1$ and $j=r$, and so we may set $Z^{\lambda_j} = V^{j}\setminus V^{j-1} \subseteq \mathbb{A}^m_{\kk}$. 
\end{proof}

\Cref{cor: constructibility for fpts} below---an immediate consequence of \Cref{prop: stratification}---is concerned with the local behavior of a polynomial at a point, which we assume to be the origin.
Recall that if $\ideala$ is an ideal in a polynomial ring contained in the homogeneous maximal ideal $\idealm$, then the $F$-pure threshold of $\ideala$ {at the origin} is simply $\ft^{\idealm}(\ideala) = \fpt(\ideala_{\idealm})$.

In order to state the corollary, we fix the following notation:
Given positive integers $n$ and $d$ and a field $\kk$, let $\mathbf{P}(n, d, \kk)$ denote the set of polynomials over $\kk$ of degree at most $d$ in $n$ variables that vanish at the origin.  
We identify this set with $\mathbb{A}_{\kk}^m$, where $m$ is the number of nonconstant monomials of degree at most $d$ in $n$ variables, which gives meaning to the notion of a locally closed subset of $\mathbf{P}(n, d, \kk)$.  
If $\kk$ has prime characteristic and is $F$-finite, let $\mathbf{F}(d,n, \kk)$ be the set of all $F$-pure thresholds at the origin of polynomials in $\mathbf{P}(n, d, \kk)$, and let
\[ \mathbf{F}(n, d, p) = \bigcup \ \mathbf{F}(n, d, \kk), \]
where the union is taken over all $F$-finite fields $\kk$ of characteristic $p>0$.

\begin{cor}\label{cor: constructibility for fpts}
   The set $\mathbf{F}(n, d, p)$ is finite.
   Moreover, given $\lambda \in \mathbf{F}(n, d, \kk)$, the set of polynomials in $\mathbf{P}(n,d,\kk)$ whose $F$-pure threshold at the origin equals $\lambda$ is a locally closed set.
   Furthermore, the defining equations of this set have coefficients in $\FF_p$, and may depend on $p$, but not on the particular field $\kk$.
\qed
\end{cor}

\begin{rem}[An analogy with the Bernstein--Sato polynomial]  
	\Cref{cor:  constructibility for fpts} is in direct analogy with known results on the Bernstein--Sato polynomial:  
	Given a field $\kk$ of characteristic zero, let $\mathbf{B}(n, d, \kk)$ be the set of all Bernstein--Sato 
		polynomials of elements of $\mathbf{P}(n, d, \kk)$, then let $\mathbf{B}(n, d)$ be the union of all 
		$\mathbf{B}(n,d, \kk)$, over all fields $\kk$ of characteristic zero.    
	Lyubeznik showed that $\mathbf{B}(n,d)$ is finite, and also asked whether the subset of $\mathbf{P}(n,d, \kk)$ 
		corresponding to some fixed polynomial in $\mathbf{B}(n,d, \kk)$ is constructible \cite{lyubeznik.bs_polys}.  
	Leykin gave a positive answer to this question, and also showed that the defining equations of the 
		constructible set have coefficients in $\QQ$, and are independent of $\kk$ \cite[Theorem~3.5]{leykin.constructibility-of-sets-of-polynomials}.
\end{rem}

\begin{rem}[On effective computability of strata]
   The proof of \Cref{prop: stratification} is constructive.
   In particular, if one could effectively compute the set $S$ in its statement, then our proof would lead to an explicit description of the sets $Z^{\lambda}$.
   This would be especially interesting
   in the context of the above corollary, 
   so that one may compare the resulting strata with those computed by Leykin in the context of Bernstein--Sato polynomials.
\end{rem}

\subsection{Sets of least critical exponents}

In this subsection, we are motivated by a result of Blickle, Musta\c{t}\u{a}, and Smith regarding sets of $F$-pure thresholds of principal ideals.
Throughout the subsection, we adopt the convention that the critical exponent of the trivial ideal, with respect to any ideal $\idealb$, is zero (consequently, $\lce(\ideal{0})=0$).
We fix a prime $p$, a positive integer $n$, and adopt the following notation:  
\begin{itemize}
   \item $\mathcal{L}_n$ is the set of all least critical exponents of proper ideals in regular $F$-finite domains of characteristic $p$ and dimension at most $n$, and $\mathcal{T}_n$ is the subset of all $F$-pure thresholds of proper principal ideals in such rings.
   \item $\mathcal{L}_n^{\circ}$ is the set of all critical exponents at the origin (that is, with respect to the homogeneous maximal ideal) of proper ideals in some polynomial ring in $n$ variables over some algebraically closed field of characteristic $p$, and let $\mathcal{T}_n^\circ$ denote the subset of $F$-pure thresholds at the origin of proper principal ideals in these polynomial rings.
\end{itemize}
The aforementioned result establishes that $\mathcal{T}_n = \overline{\mathcal{T}_n^{\circ}}$, where the right-hand side denotes the closure of $\mathcal{T}^\circ_n$ \cite[Theorem~1.2]{blickle+mustata+smith.F-thresholds_hyper}. 
We prove the following analog for least critical exponents.  

\begin{thm} \label{thm: sets of crits and fpts}
$\mathcal{L}_n = \overline{\mathcal{L}_n^{\circ}}$.
\end{thm}

In order to prove \Cref{thm: sets of crits and fpts}, we need the following lemma.

\begin{lem} \label{lem: bound on difference of lces}
   If $\ideala$ and $\idealb$ are proper ideals of an $F$-finite regular local ring $(R, \idealm)$ and $\ideala + \frob{\idealm}{q} = \idealb + \frob{\idealm}{q}$, then $\left| \lce(\ideala) - \lce(\idealb) \right| \le 1/q$.
\end{lem}
 
\begin{proof}
   The containment $\idealb \subseteq \ideala + \frob{\idealm}{q}$ and \Cref{lem: bound on sum of crits} show that
   \[\lce(\idealb) \le \lce\big(\ideala + \frob{\idealm}{q}\big) \leq \lce(\ideala) + \lce\big(\frob{\idealm}{q}\big) = \lce(\ideala) + \frac1q.\]
   Our claim follows from this, and the analogous statement when we reverse the roles of $\ideala$ and $\idealb$.
\end{proof}

We are now ready to prove \Cref{thm: sets of crits and fpts}.

\begin{proof}[Proof of \Cref{thm: sets of crits and fpts}]
   Given a proper ideal $\ideala$ of a regular $F$-finite domain $R$ of characteristic $p$, any maximal ideal $\idealm$ containing $\frob{\ideala}{\lce(\ideala)}$ realizes $\lce(\ideala)$, by the comments preceding \Cref{defn: realizes}, and $\lce(\ideala) = \crit^{\idealm}(\ideala) = \crit^{\idealm R_\idealm}(\ideala R_\idealm) = \crit^{\idealm \widehat{R}_\idealm}(\ideala \widehat{R}_\idealm)$, by \Cref{prop:  crit can be computed locally} and the fact that Frobenius powers commute with completion.
   Then by the Cohen Structure Theorem, $\mathcal{L}_n$ is the set of all least critical exponents of proper ideals in power series rings $\kk\llbracket  \vx \rrbracket = \kk\llbracket x_1, \ldots, x_r \rrbracket$, where $r \leq n$, and $\kk$ an $F$-finite field of characteristic $p$.

   Fix one such ring $\kk\llbracket \vx \rrbracket$; let $\idealm$ be its maximal ideal, and $\ideala$ an arbitrary proper ideal.
   Given $f \in \kk\llbracket \vx \rrbracket$ and $d \in \NN$, let $f_{\leq d}$ denote the truncation of $f$ up to degree $d$ (which can be considered in $\kk[\vx]$).
   Fix generators $f_1, \ldots, f_t$ of $\ideala$, and define $\ideala_{\leq d} = \ideal{(f_1)_{\leq d}, \ldots, (f_t)_{\leq d}} \subseteq \kk\llbracket \vx \rrbracket$.
   By the pigeonhole principle, $\ideala + \frob{\idealm}{q} = \ideala_{\leq d} + \frob{\idealm}{q}$ for $d \geq r(q-1)$.
   Hence the limit of the sequence $\seq{\lce(\ideala_{\leq d})}$ is $\lce(\ideala)$ by \Cref{lem: bound on difference of lces}.

   If $\ideal{\vx} = \ideal{x_1, \ldots, x_r}$ and $\ideala_d = \ideal{(f_1)_{\leq d}, \ldots, (f_t)_{\leq d}}$ are ideals of $\kk[x]$, we know that $\lce(\ideala_{\leq d}) = \crit^\idealm(\ideala_{\leq d}) = \crit^{\ideal{\vx}}(\ideala_d)$, again because Frobenius powers commute with completion.
   Since the sequence whose limit defines $\crit^{\ideal{\vx}}(\ideala_d)$ agrees with the sequence defining the critical exponent with respect to $\ideal{\vx}$ of the extension of $\ideala_d$ to the polynomial ring in $n$ variables over the algebraic closure of $\kk$, these critical exponents are equal.
   Hence we conclude that $\mathcal{L}_n \subseteq \overline{\mathcal{L}_n^\circ}$.

   Since $\overline{\mathcal{T}_n^\circ} = \mathcal{T}_n$ \cite[Theorem~1.2]{blickle+mustata+smith.F-thresholds_hyper}, and $\mathcal{L}_n^\circ = \mathcal{T}_n^\circ$ by \Cref{cor: crit(a) = max(c(g): g in a)}, we conclude that $\mathcal{L}_n \subseteq \overline{\mathcal{L}_n^\circ} = \overline{\mathcal{T}_n^\circ} = \mathcal{T}_n \subseteq \mathcal{L}_n$, so that equality holds throughout.
\end{proof}

\begin{rem} Note that the proof of \Cref{thm: sets of crits and fpts} yields the stronger statement \[ \mathcal{L}_n = \overline{\mathcal{L}_n^\circ} = \overline{\mathcal{T}_n^\circ} = \mathcal{T}_n. \]  
\end{rem}

\section{Behavior as $p \to \infty$} \label{s: behavior as p approaches infinity}

In this section, we view Frobenius powers and critical exponents from the point of view of reduction to prime characteristic,  
	with an eye towards establishing connections with certain invariants from birational geometry.  
To simplify our discussion, we only consider the local behavior of varieties in $\mathbb{A}^n_{\QQ}$ at a point, 
	which we may assume to be the origin.  
Algebraically, this means that we focus on nonzero ideals in the localization of $\QQ[\vx] = \QQ[x_1, \ldots, x_n]$ at $\idealm = \ideal{ x_1, \ldots, x_n }$.  

Recall that the \emph{multiplier ideal} of $\ideala \subseteq \QQ[\vx]_{\idealm}$ with respect to a nonnegative real parameter $t$ is the ideal in $\QQ[\vx]_{\idealm}$ given by 
\[ \mathcal{J}(\ideala^{t}) =  \{ h \in \QQ[\vx]_{\idealm} : \operatorname{div}(\pi^{\ast} h) \geq   t Z - K_{\pi} \},\] 
where $\pi: X \to \operatorname{Spec}( \QQ[\vx]_{\idealm})$ is a log resolution of $\ideala$ with relative canonical divisor $K_{\pi}$, and $\ideala \cdot \mathcal{O}_X = \mathcal{O}_X(-Z)$.   
We also recall that the \emph{log canonical threshold} of $\ideala$, denoted $\lct(\ideala)$, is the supremum of all 
	$t > 0$ such that $\mathcal{J}(\ideala^{t})$ is the unit ideal.
We refer the reader to \cite{blickle+lazarsfeld.intro_multiplier_ideals} for the basics of multiplier ideals and log canonical thresholds. 

\begin{rem}[Reduction modulo a prime integer]
	Let $\idealb \subseteq\QQ[\vx]_{\idealm}$ be an ideal generated by a set $\mathcal{B}$ consisting 
		of polynomials with integer coefficients.  
	Given a prime $p$, the \emph{reduction of $\idealb$ modulo} $p$ is the ideal $\idealb_p$ generated by the image of $\mathcal{B}$
		in $\FF_p[\vx]_{\idealm_p}$, where $\idealm_p$ is the homogeneous maximal ideal of $\FF_p[\vx]$.  
	Note that, while the reductions $\idealb_p$ may depend on the choice of $\mathcal{B}$, any two such choices will yield 
		identical reductions once $p \gg 0$. 
\end{rem}

Our interest in the multiplier ideals of an ideal  $\ideala \subseteq \QQ[\vx]_{\idealm}$ stems from the fact 
	that they are ``universal test ideals" of the reductions $\ideala_p$.  
Indeed, it was shown in \cite{smith.mult-ideal-is-univ-test-ideal,hara+yoshida.generalization_TC_multiplier_ideals} 
	that if $p \gg 0$, then $\tau(\ideala_p^{t})$ is contained in $\mathcal{J}(\ideala^{t})_p$ for all $t > 0$.  
Furthermore, if $t>0$ is fixed, then the ideals $\tau(\ideala_p^{t})$ and $\mathcal{J}(\ideala^{t})_p$ agree whenever 
	$p \gg 0$, with the precise bound on $p$ depending on $t$.  
    
We record an important corollary of this remarkable result below.

\begin{rem}\label{limit of F-thresholds: R}	
	If $\ideala \subseteq \idealb$ are nonzero proper ideals in $\QQ[\vx]_{\idealm}$, then 
		\[  \lim_{p \to \infty} \ftabp \] 
		exists, and equals $\min \{ t \in\RRpos : \mathcal{J}(\ideala^{t}) \subseteq \idealb \}$.   
	Indeed, if we denote this minimum by $\lambda$, then $\mathcal{J}(\ideala^{\lambda}) \subseteq \idealb$, 
		and we may choose an integer $N_{\lambda}$ such that 
		\[ \tau(\ideala_p^{\lambda}) = \mathcal{J}(\ideala^{\lambda})_p \subseteq \idealb_p \] 
		for all $p \geq N_{\lambda}$, which shows that the terms of the sequence $\ftabp$ are at most 
		$\lambda$ whenever $p \gg 0$.  
	Next, fix $0 < t < \lambda$, so that $\mathcal{J}(\ideala^{t}) \not \subseteq \idealb$, 
		which we restate in terms of quotients as $( \mathcal{J}(\ideala^{t}) + \idealb ) / \idealb \neq 0$.  
	It is well known (see, \eg \cite{hochster+huneke.tight-closure-in-char-0}) that the property of being nonzero is preserved 
		under reduction mod $p \gg 0$, and so there exists an integer $N_{t}$ such that 
		$\tau(\ideala^{t}_p) = \mathcal{J}(\ideala^{t})_p \not \subseteq \idealb_p$ for all $p \geq N_{t}$.  
	In particular, $\ftabp > t$ if $p \geq N_{t}$, which allows us to conclude that 
		\[ \lim_{p \to \infty} \ftabp = \lambda, \]
		as desired.   
	We also note the following special case of this result:  
	If $\idealb$ is the maximal ideal of $\QQ[\vx]_{\idealm}$, then $\fpt(\ideala_p) = \ftabp$, and so
		\[ \lim_{p \to \infty} \fpt(\ideala_p) = \lct(\ideala).\]  
\end{rem}

In the remainder of this section, we establish analogous results for Frobenius powers and critical exponents.

\begin{thm}\label{limits of Frobenius powers: T}
	Suppose $\ideala$ is an ideal of $\QQ[x_1, \ldots, x_n]_{\idealm}$.  
	If $0 < t < 1$, then 
		\[ \frob{\ideala_p}{t} = \tau(\ideala_p^{t}) = \mathcal{J}(\ideala^{t})_p \] 
		for all $p \gg 0$, with the precise bound on $p$ depending on $t$.
\end{thm}

\begin{proof}  
	Fix $0 < t < 1$.  
	Like test ideals and Frobenius powers, the multiplier ideals of $\ideala$ are locally constant to the right, and 
		so we may fix $0 < \epsilon \ll 1$ with $\mathcal{J}(\ideala^{t}) = \mathcal{J}(\ideala^{t + \epsilon})$.  
	Having fixed both $t$ and $\epsilon$,  it follows that 
		\begin{equation}\label{chain: e}
			\tau(\ideala^{t}_p) = \mathcal{J}(\ideala^{t})_p 
				= \mathcal{J}(\ideala^{t + \epsilon})_p = \tau(\ideala^{t + \epsilon}_p)
		\end{equation}
		for all $p \gg 0$.  
	However, if $\ideala$ (and hence each $\ideala_p$) can be generated by $m$ elements, then $\frac{m-1}{p-1} < \epsilon$ 
		for all $p \gg 0$, and so \Cref{uniform relation between ideals: P}  tells us that 
		\[ \tau(\ideala_p^{t + \epsilon}) \subseteq \frob{\ideala_p}{t} \subseteq \tau(\ideala_p^{t})\] 
		for all $p \gg 0$.  
	Finally, comparing this with \eqref{chain: e} we see that any two ideals appearing in either of these chains must agree.
\end{proof}

\begin{thm}\label{limits of critical exponents: T}
  	Consider nonzero proper ideals $\ideala \subseteq \idealb$ in $\QQ[x_1, \ldots, x_n]_{\idealm}$, and set 
		\[  \lambda = \lim_{p \to \infty} \ftabp = \min \{ t \in\RRpos : \mathcal{J}(\ideala^{t}) \subseteq \idealb \}. \]
	If $\lambda > 1$, then $\critabp =1$ for all $p \gg 0$.  
	Otherwise, 
		\[  \lim_{p \to \infty} \critabp = \lambda. \]
\end{thm}

\begin{proof}  
	Suppose that $\ideala$ (and hence all $\ideala_p$) can be generated by $m$ elements.  
	If $\lambda > 1$, yet $\critabp \neq 1$ for infinitely many $p$, then \Cref{uniform nu and mu comparison: C} implies that 
		a subsequence of terms $\critabp$ converges to $\lambda$, which is impossible, since the containment 
		$\ideala \subseteq \idealb$ implies that each such term is at most $1$.  
	If $\lambda = 1$, then \Cref{uniform nu and mu comparison: C} tells us that either $\critabp = 1$, 
		or $\crit^{\idealb_p}(\ideala_p)$ lies between two sequences that converge to $\lambda = 1$, 
		and so $\lim_{p\to \infty} \crit^{\idealb_p}(\ideala_p)  = 1$.  
	Finally, if $\lambda < 1$, then $\critabp \leq \ftabp < 1$ for all $p \gg 0$, and we may once again invoke 
		\Cref{uniform nu and mu comparison: C} to see that $\lim_{p\to \infty} \crit^{\idealb_p}(\ideala_p)  = \lambda$.
\end{proof}

\begin{rem} 
	\Cref{limits of critical exponents: T} implies that if $\ideala \subseteq \idealb$ are nonzero proper ideals of $ \QQ[\vx]_{\idealm}$, then
		\[ \lim_{p \to \infty} \critabp=\min \Big\{ 1, \lim_{p\to\infty} \ftabp \Big\}.\] 
	Alternatively, this also follows from \Cref{limits of Frobenius powers: T}, and an argument modeled on 
		that given in \Cref{limit of F-thresholds: R}.  
	In any case, if $\idealb$ is the maximal ideal of $\QQ[\vx]_{\idealm}$, then $\lce(\ideala_p) = \critabp$, and so 
		\[ \lim_{p \to \infty} \lce(\ideala_p) = \min \{ 1, \lct(\ideala) \}.\] 
\end{rem}

The following is motivated by similar well-known questions surrounding the relationship between test ideals and multiplier ideals.

\begin{quest}  	Given a nonzero ideal $\ideala$ in $\QQ[\vx]_{\idealm}$ with $\lct(\ideala)$ at most $1$, do there exist infinitely many primes $p$ such that $\lce(\ideala_p) = \lct(\ideala)$?  More generally, does there exist an infinite set of primes $X$ for which $\ideala_p^{[t]} = \mathcal{J}(\ideala^t)_p$ for all $p \in X$ and $t$ in the open unit interval?
\end{quest}

\section{Computing Frobenius powers} \label{s: algorithm}

In this section, we outline an algorithm that can be implemented to compute arbitrary rational Frobenius powers 
	of ideals in polynomial rings.

Before proceeding, we point out that it is possible to effectively compute $[1/q]^\mathrm{th}$ powers in
	a polynomial ring over a finite field.
This has been implemented in the \emph{Macaulay2} \cite{M2} package \emph{TestIdeals} \cite{testideals} (see also its companion article \cite{TestIdealsPaper}), and relies on the following
	proposition, which describes $\frob{\ideala}{1/q}$ in terms of the generators of $\ideala$.

\begin{prop}[{\cite[Proposition~2.5]{blickle+mustata+smith.discr_rat_FPTs}}]\label{prop: roots in terms of generators}
	Suppose $R$ is free over $R^q$, with basis $\{e_1,\ldots,e_n\}$, and let $\ideala=\ideal{f_1,\ldots,f_s}$ be an ideal of $R$.
	Write each $f_i$ as an $R^q$-linear combination 	of the basis $\{e_1,\ldots,e_n\}$ as follows\textup:
		\[f_i=\sum_{j=1}^n a_{ij}^q e_j,\]
	 	where $a_{ij}\in R$.
	Then $\frob{\ideala}{1/q}=\ideal{a_{ij} |1\le i\le s, 1\le j\le n}$. 
\qed
\end{prop}

Let $R$ be a polynomial ring over a finite field of characteristic~$p$.
      Let $\ideala$ be an ideal in $R$, and $t$ a nonnegative rational number.
      We now describe an algorithm for computing $\frob{\ideala}{t}$.

\begin{itemize}
\item If $t$ is of the form $k/p^b$ for some $b\in \NN$, then 
		\[\frob{\ideala}{t}=\frob{\big(\frob{\ideala}{k}\big)}{1/p^b},\] 
		which can be computed using \Cref{prop: roots in terms of generators}.
\item	Otherwise, $t$ can be written in the form $t=k/(p^b(p^c-1))$, for some $c>0$.
\item Set $u=k/(p^c-1)$, and note that, by \Cref{prop: basic properties of real powers}\iref{item: real power of real power}, 
	\[\frob{\ideala}{t}=\frob{\big(\frob{\ideala}{u}\big)}{1/p^b}.\]
\item Write $k = (p^c-1)l+r$, with $0 \le r < p^c-1$, and set $v= r/(p^c-1)$,
		so that $u=l+v$, and $\frob{\ideala}{u}=\frob{\ideala}{l}\frob{\ideala}{v}$ by \Cref{cor: skoda}.
\item For each $e\ge 1$, set
	\begin{equation*}
		v_e = \frac{r}{p^c}+\frac{r}{p^{2c}}+\cdots +\frac{r}{p^{ec}}+\frac{1}{p^{ec}}
		= \frac{rp^{(e-1)c}+rp^{(e-2)c}+\cdots+r+1}{p^{ec}}.
	\end{equation*}
	Note that $v_e\searrow v$, so that $\frob{\ideala}{v}=\frob{\ideala}{v_e}$ for $e\gg 0$.
	It remains to determine how large $e$ needs to be.
\item Let 
		\[\mu=\min\{e\ge 1: \frob{\ideala}{v_e}=\frob{\ideala}{v_{e+1}}\}.\]
	To compute $\mu$ efficiently, note that the ideals $ \frob{\ideala}{v_e}$ can be computed recursively.
	Indeed, let $n_e$ denote the numerator of $v_e$, that is, $n_e=rp^{(e-1)c}+rp^{(e-2)c}+\cdots+r+1$.
	Then
		\[\frob{\ideala}{v_{e+1}}=\frob{\big(\frob{\ideala}{n_{e+1}}\big)}{1/p^{(e+1)c}}=\frob{\big(\frob{\ideala}{n_e+r p^{ec}}\big)}{1/p^{(e+1)c}}.\]
	Since the supporting base $p$ digits of $n_e$ and $r p^{ec}$ are disjoint, 
		$\frob{\ideala}{n_e+r p^{ec}}=\frob{\ideala}{n_e}\frob{\ideala}{r p^{ec}}$ by \Cref{prop: basic properties or integer powers}\iref{item: products of powers}, so
		\begin{align*}
			\frob{\ideala}{v_{e+1}}&=\frob{\big(\frob{\ideala}{n_e}\frob{\ideala}{r p^{ec}}\big)}{1/p^{(e+1)c}}\\
			&= \frob{\Big(\frob{\big(\frob{\ideala}{n_e}\frob{\ideala}{r p^{ec}}\big)}{1/p^{ec}}\Big)}{1/p^c} &
				\text{by \Cref{lem: basic facts of roots}\iref{item: root of root}}\\
			&= \frob{\Big(\frob{\big(\frob{\ideala}{n_e}\frob{(\frob{\ideala}{r})}{p^{ec}}\big)}{1/p^{ec}}\Big)}{1/p^c}
				&\text{by \Cref{prop: basic properties or integer powers}\iref{item: powers of powers}}\\
			&= 
			\frob{
				\Big(
					\frob{\ideala}{r}
					\frob{
						\big(
							\frob{\ideala}{n_e}
						\big)
					}{1/p^{ec}}
				\Big)
			}{1/p^c}
			&\text{by \Cref{lem: qth roots vs. products}}\\
			&= 
			\frob{
				\Big(
					\frob{\ideala}{r}
					\frob{\ideala}{v_e}
				\Big)
			}{1/p^c}.
		\end{align*}
\item The above  recursion shows that $\frob{\ideala}{v_e}=\frob{\ideala}{v_\mu}$, for all 
			$e\ge \mu$, so $\frob{\ideala}{v}=\frob{\ideala}{v_\mu}$.
\item In conclusion, 
		\[\frob{\ideala}{t}=\frob{\big(\frob{\ideala}{l}\frob{\ideala}{v_\mu}\big)}{1/p^b}.\]
\end{itemize}

\noindent Pseudocode for the above algorithm is provided below.

\begin{algorithm}[h]
\DontPrintSemicolon
\Indp
\medskip
\SetKwInput{Input}{Input}\SetKwInput{Output}{Output}
 \Input{$\ideala$, an ideal in a polynomial ring over a finite field\\  \hskip 12.7mm $t$, a nonnegative rational number}
\medskip
 \Output{$\frob{\ideala}{t}$}
\medskip
\lIf{\emph{$t=k/p^b$, for some $k,b\in \NN$}}{\KwRet{$\frob{\big(\frob{\ideala}{k}\big)}{1/p^b}$}}
Write $t=\frac{k}{p^b(p^c-1)}$, with $b\in \NN$ and $k,c\in \NNpos$\;
Divide $k$ by $p^c-1$: $k = (p^c-1)l+r$, where $0 \le r < p^c-1$\;
$\idealc\leftarrow\frob{\big(\frob{\ideala}{r+1}\big)}{1/p^c}$\;
\Repeat{$\idealc\subseteq \idealb$}{
$\idealb\leftarrow \idealc$\;
$\idealc\leftarrow \frob{\big(\frob{\ideala}{r}\idealb\big)}{1/p^c}$\;
}
\KwRet{ $\frob{\big(\frob{\ideala}{l}\idealb\big)}{1/p^b}$}
\caption{Computing rational Frobenius powers}\label{alg}
\medskip
\end{algorithm}

\section*{Acknowledgements}
The authors thank Josep \`{A}lvarez Montaner, Anton Leykin, and Linquan Ma, whose questions helped shape the paper, especially \Cref{s: principal principle}.
Thanks also go to the anonymous referee, for the valuable comments, suggestions, and corrections. 
Finally, the first and third authors are grateful to the National Science Foundation for support of this research through Grants DMS-1600702 and DMS-1623035, respectively.
 
{\small
\bibliographystyle{amsalpha}
\bibliography{bibdatabase}

\newcommand{\etalchar}[1]{$^{#1}$}
\providecommand{\bysame}{\leavevmode\hbox to3em{\hrulefill}\thinspace}
\providecommand{\MR}{\relax\ifhmode\unskip\space\fi MR }
\providecommand{\MRhref}[2]{%
  \href{http://www.ams.org/mathscinet-getitem?mr=#1}{#2}
}
\providecommand{\href}[2]{#2}
\begin{thebibliography}{DSNBP18}

\bibitem[BBB{\etalchar{+}}]{testideals}
E.~Bela, A.~F. Boix, J.~Bruce, D.~Ellingson, D.~J. Hern\'andez, Z.~Kadyrsizova,
  M.~Katzman, S.~Malec, M.~Mastroeni, M.~Mostafazadeh\-fard, M.~Robinson,
  K.~Schwede, D.~Smolkin, P.~Teixeira, and E.~E. Witt, \emph{{TestIdeals: a
  package for calculations of singularities in positive characteristic.
  Version~1.01}}, available at
  \url{https://github.com/Macaulay2/M2/tree/master/M2/Macaulay2/packages/TestIdeals}.

\bibitem[BHK{\etalchar{+}}18]{TestIdealsPaper}
A.~F. Boix, D.~J. Hern\'andez, Z.~Kadyrsizova, M.~Katzman, S.~Malec,
  M.~Robinson, K.~Schwede, D.~Smolkin, P.~Teixeira, and E.~E. Witt, \emph{{The
  \emph{TestIdeals} package for \emph{Macaulay2}}}, preprint,
  \href{https://arxiv.org/abs/1810.02770}{arXiv:1810.02770 [math.AC]}, 2018.

\bibitem[BL04]{blickle+lazarsfeld.intro_multiplier_ideals}
M.~Blickle and R.~Lazarsfeld, \emph{An informal introduction to multiplier
  ideals}, Trends in {C}ommutative {A}lgebra, Math.\ Sci.\ Res.\ Inst.\ Publ.,
  vol.~51, Cambridge Univ.\ Press, Cambridge, 2004, pp.~87--114.

\bibitem[BMS08]{blickle+mustata+smith.discr_rat_FPTs}
M.~Blickle, M.~Musta{\c{t}}\u{a}, and K.~E. Smith, \emph{Discreteness and
  rationality of {$F$}-thresholds}, Michigan Math.\ J. \textbf{57} (2008),
  43--61.

\bibitem[BMS09]{blickle+mustata+smith.F-thresholds_hyper}
\bysame, \emph{{$F$}-thresholds of hypersurfaces}, Trans.\ Amer.\ Math.\ Soc.
  \textbf{361} (2009), no.~12, 6549--6566.

\bibitem[Dic02]{dickson.multinomial}
L.~E. Dickson, \emph{Theorems on the residues of multinomial coefficients with
  respect to a prime modulus}, Quart.\ J.\ Pure Appl.\ Math. \textbf{33}
  (1902), 378--384.

\bibitem[DSNBP18]{stefani+betancourt+perez.existence-of-FTs}
A.~De~Stefani, L.~N{\'u}{\~n}ez-Betancourt, and F.~P{\'e}rez, \emph{On the
  existence of {$F$}-thresholds and related limits}, Trans. Amer. Math. Soc.
  \textbf{370} (2018), no.~9, 6629--6650.

\bibitem[GS]{M2}
D.~R. Grayson and M.~E. Stillman, \emph{Macaulay2, a software system for
  research in algebraic geometry}, available at
  \url{http://www.math.uiuc.edu/Macaulay2/}.

\bibitem[Her12]{hernandez.F-purity_of_hypersurfaces}
D.~J. Hern\'andez, \emph{{$F$}-purity of hypersurfaces}, Math.\ Res.\ Lett.
  \textbf{19} (2012), no.~2, 389--401.

\bibitem[Her15]{hernandez.diag_hypersurf}
\bysame, \emph{{$F$}-invariants of diagonal hypersurfaces}, Proc.\ Amer.\
  Math.\ Soc. \textbf{143} (2015), no.~1, 87--104.

\bibitem[HH99]{hochster+huneke.tight-closure-in-char-0}
M.~Hochster and C.~Huneke, \emph{Tight closure in equal characteristic zero},
  available at \url{http://www.math.lsa.umich.edu/~hochster/tcz.pdf}, 1999.

\bibitem[HNBW17]{hernandez+etal.local_m-adic_constancy}
D.~J. Hern\'andez, L.~N{\'u}{\~n}ez-Betancourt, and E.~E. Witt, \emph{Local
  $\mathfrak{m}$-adic constancy of {$F$}-pure thresholds and test ideals},
  Math.\ Proc.\ Cambridge\ Philos.\ Soc. (2017), 1--11.

\bibitem[How01]{howald.multiplier_ideals_of_monomial_ideals}
J.~Howald, \emph{Multiplier ideals of monomial ideals}, Trans.\ Amer.\ Math.\
  Soc. \textbf{353} (2001), no.~7, 2665--2671.

\bibitem[How03]{howald.multiplier_ideals_of_generic_polynomials}
\bysame, \emph{Multiplier ideals of sufficiently general polynomials},
  preprint, \href{https://arxiv.org/abs/math/0303203}{arXiv:0303203 [math.AG]},
  2003.

\bibitem[HTW19]{hernandez+etal.frobenius_examples}
D.~J. Hern\'andez, P.~Teixeira, and E.~E. Witt, \emph{Frobenius powers of some
  monomial ideals}, to appear in J.\ Pure Appl.\ Algebra,
  \url{https://doi.org/10.1016/j.jpaa.2019.04.015}, 2019.

\bibitem[HY03]{hara+yoshida.generalization_TC_multiplier_ideals}
N.~Hara and {K.-i.} Yoshida, \emph{A generalization of tight closure and
  multiplier ideals}, Trans.\ Amer.\ Math.\ Soc. \textbf{355} (2003), no.~8,
  3143--3174.

\bibitem[Laz04]{lazarsfeld.positivity-II}
R.~Lazarsfeld, \emph{Positivity in {A}lgebraic {G}eometry {II}},
  Springer--Verlag, Berlin, 2004.

\bibitem[Ley01]{leykin.constructibility-of-sets-of-polynomials}
A.~Leykin, \emph{Constructibility of the set of polynomials with a fixed
  {B}ernstein--{S}ato polynomial: an algorithmic approach}, J.\ Symbolic
  Comput. \textbf{32} (2001), no.~6, 663--675.

\bibitem[Lyu97]{lyubeznik.bs_polys}
G.~Lyubeznik, \emph{On {B}ernstein--{S}ato polynomials}, Proc.\ Amer.\ Math.\
  Soc. \textbf{125} (1997), no.~7, 1941--1944.

\bibitem[MTW05]{mustata+takagi+watanabe.F-thresholds}
M.~Musta{\c{t}}\u{a}, S.~Takagi, and {K.-i.} Watanabe, \emph{{F}-thresholds and
  {B}ernstein--{S}ato polynomials}, European {C}ongress of {M}athematics
  (Z{\"u}rich), Eur.\ Math.\ Soc., 2005, pp.~341--364.

\bibitem[Smi00]{smith.mult-ideal-is-univ-test-ideal}
K.~E. Smith, \emph{The multiplier ideal is a universal test ideal}, Comm.\
  {A}lgebra \textbf{28} (2000), no.~12, 5915--5929.

\bibitem[TW04]{takagi+watanabe.F-pure_thresholds}
S.~Takagi and {K.-i.} Watanabe, \emph{On {F}-pure thresholds}, J.\ Algebra
  \textbf{282} (2004), no.~1, 278--297.

\end{thebibliography}
}	
\end{document}